\documentclass{amsart}

\usepackage{amsfonts,amssymb,amsmath,enumerate,verbatim,mathtools,tikz,bm,mathrsfs,tikz-cd,hyperref,comment,stmaryrd,amsthm}
\usepackage[shortlabels]{enumitem}
\usepackage{colonequals}
\usepackage{cleveref}
\usepackage{adjustbox}
\usepackage[all,pdf]{xy}
\hypersetup{
    colorlinks=true,
    linkcolor=blue,
    filecolor=magenta,      
    urlcolor=cyan,
}
\usetikzlibrary{arrows}
\usetikzlibrary{decorations.markings}
\tikzset{negated/.style={
        decoration={markings,
            mark= at position 0.5 with {
                \node[transform shape] (tempnode) {$\backslash$};
            }
        },
        postaction={decorate}
    }
}

\makeatletter
\@namedef{subjclassname@2020}{\textup{2020} Mathematics Subject Classification}
\makeatother

\author{Souvik Dey}
\address[SD]{Department of Mathematical Sciences, 850 West Dickson Street, University of Arkansas, Fayetteville, Arkansas 72701}
\email{souvikd@uark.edu}    

\author[Kaito Kimura]{Kaito Kimura}
\address[KK]{Graduate School of Mathematics, Nagoya University, Furocho, Chikusaku, Nagoya 464-8602, Japan}
\email{m21018b@math.nagoya-u.ac.jp}

\author[Jian Liu]{Jian Liu}
\address[JL]{School of Mathematics and Statistics, and Hubei Key Laboratory of Mathematical Sciences,  Central China Normal University,  Wuhan 430079, P.R. China}
\email{jianliu@ccnu.edu.cn}

\author[Yuya Otake]{Yuya Otake}
\address[YO]{Graduate School of Mathematics, Nagoya University, Furocho, Chikusaku, Nagoya 464-8602, Japan}
\email{m21012v@math.nagoya-u.ac.jp}


\def \S {\mathsf{S}}

\DeclareMathOperator{\embdim}{embdim}

\DeclareMathOperator{\sg}{sg}
\DeclareMathOperator{\depth}{depth}

\DeclareMathOperator{\h}{H}

\newcommand{\syz}{\Omega}

\newcommand{\Z}{\mathbb{Z}}

\newcommand{\T}{\mathcal{T}}

\newcommand{\Y}{\mathcal Y}

\newcommand{\D}{\mathsf{D}}

\pgfdeclarelayer{bg}    % declare background layer
\pgfsetlayers{bg,main} 
\newcommand{\x}{{\bm{x}}}

\newcommand{\m}{\mathfrak{m}}

\newcommand{\p}{\mathfrak{p}}

\DeclareMathOperator{\Ima}{Im}
\DeclareMathOperator{\Ker}{Ker}

\DeclareMathOperator{\pd}{pd}

\DeclareMathOperator{\add}{add}

\DeclareMathOperator{\ca}{ca}

\DeclareMathOperator{\Spec}{Spec}

\DeclareMathOperator{\Sing}{Sing}
\DeclareMathOperator{\thick}{thick}
\DeclareMathOperator{\GProj}{\mathsf{GProj}}
\DeclareMathOperator{\GInj}{\mathsf{GInj}}
\DeclareMathOperator{\Gproj}{\mathsf{Gproj}}

\DeclareMathOperator{\Proj}{\mathsf{Proj}}
\DeclareMathOperator{\Inj}{Inj}

\DeclareMathOperator{\proj}{\mathsf{proj}}
\DeclareMathOperator{\FPD}{FPD}
\DeclareMathOperator{\fl}{\mathsf{fl}}

\def\add{\operatorname{add}}

\def\ann{\operatorname{ann}}

\def\ca{\mathsf{ca}}
\def\CM{\mathsf{CM}}
\def\OCM{\Omega\mathsf{CM}}

\def\cok{\operatorname{Coker}}
\def\depth{\operatorname{depth}}

\def\Ext{\operatorname{Ext}}

\def\h{\operatorname{H}}
\def\Hom{\operatorname{Hom}}

\def\m{\mathfrak{m}}

\def\mod{\operatorname{mod}}
\def\Mod{\operatorname{Mod}}

\def\p{\mathfrak{p}}
\def\res{\operatorname{res}}

\def\spec{\operatorname{Spec}}

\def\Tor{\operatorname{Tor}}
\def\Tr{\operatorname{Tr}}
\def\Tr{\operatorname{Tr}}

\def\V{\operatorname{V}}
\def\X{\mathcal{X}}

\def\TF{\mathsf{TF}} 
\def\NF{\operatorname{NF}} 
\def\Supp{\operatorname{Supp}}

\newtheorem{theorem}{Theorem}[section]
\newtheorem{prop}[theorem]{Proposition}
\newtheorem{thm}[theorem]{Theorem}

\newtheorem{lem}[theorem]{Lemma}
\newtheorem{cor}[theorem]{Corollary}

\theoremstyle{definition}
\newtheorem{ex}[theorem]{Example}
\newtheorem{rem}[theorem]{Remark}

\newtheorem{chunk}[theorem]{}

\usepackage[utf8]{inputenc}

\newtheorem*{ack}{Acknowledgements}

\title[On local rings of finite syzygy representation type]{On local rings of finite syzygy representation type}
\begin{document}
\maketitle
\begin{abstract}
Let $R$ be a commutative Noetherian local ring. We characterize when its completion has an isolated singularity, thereby strengthening the Dao–Takahashi refinement of the Auslander–Huneke–Leuschke–Wiegand theorem.  We investigate the ascent and descent of finite and countable syzygy representation type along the canonical map from $R$ to its completion. One consequence is a complete affirmative answer to Schreyer’s conjecture. We explore analogues of Chen’s questions in the context of finite Cohen--Macaulay representation type over Cohen--Macaulay rings. The main result in this direction shows that if $R$ is Cohen--Macaulay and there are only finitely many non-isomorphic indecomposable maximal Cohen--Macaulay modules that are locally free on the punctured spectrum, then either $R$ is a hypersurface or every Gorenstein projective module is projective; moreover, every Gorenstein projective module over the completion of $R$ is a direct sum of finite generated ones. Finally, we study dominant local rings, introduced by Takahashi, under certain finite representation type conditions, and identify a new class of virtually Gorenstein rings.
\end{abstract}

\section{Introduction}
Cohen--Macaulay representation theory, which studies maximal Cohen--Macaulay modules over Cohen--Macaulay rings, began in the 1970s.  A fundamental problem is to determine when a ring has \emph{finite Cohen--Macaulay representation type}, that is, when there exist only finitely many non-isomorphic indecomposable maximal Cohen--Macaulay modules. 

For a Cohen--Macaulay local ring $(R,\m)$,
Auslander \cite{Auslander:1986} established that if $R$ is complete and of finite Cohen--Macaulay representation type, then it has an isolated singularity. Leuschke and Wiegand \cite{LW2000} extended this result to excellent local rings. Finally, Huneke and Leuschke \cite{HL} removed the assumption of excellence, showing that the conclusion holds for all Cohen--Macaulay local rings. Herzog \cite{Herzog} observed that a complete Gorenstein local ring with finite Cohen--Macaulay representation type must be a hypersurface. 
In the 1980s, under mild assumptions,  complete Gorenstein local rings having finite Cohen--Macaulay representation type were classified as simple hypersurface singularities by Buchweitz, Greuel, and Schreyer \cite{BGS},  and K\"{n}orrer \cite{Kno}. 

In this paper, let $\mod(R)$ denote the category of finitely generated $R$-modules, $\CM(R)$ its full subcategory of maximal Cohen–Macaulay modules, and $\CM_0(R)$ the full subcategory of $\CM(R)$ consisting of modules locally free on the punctured spectrum; see \Cref{MCM}.
A full subcategory $\X $ of $\mod(R)$ is said to have \emph{finite} (resp. \emph{countable}) representation type if it contains only finitely (resp. countably) many non-isomorphic indecomposable modules. Note that $R$ has finite Cohen--Macaulay type precisely means that $\CM(R)$ has finite representation type.

Recently, Dao and Takahashi \cite{DT2015}  
 introduced the notion of the dimension of a subcategory of an Abelian category with enough projective objects.  
In the Cohen–Macaulay setting, they proved that $R$ has an isolated singularity whenever $\CM_0(R)$ has finite dimension, and that the converse holds when $R$ is moreover equicharacteristic and complete. As an application, they improved the Auslander–Huneke–Leuschke–Wiegand theorem by showing that if $\CM_0(R)$ has finite representation type, then $R$ has an isolated singularity.

The first statement of \Cref{C1} strengthens the improved version of the Auslander–Huneke–Leuschke–Wiegand theorem obtained by Dao and Takahashi \cite{DT2015}, and it follows from \Cref{equivalent conditions}, which characterizes when the $\m$-adic completion $\widehat{R}$ has an isolated singularity. The second statement of \Cref{C1} characterizes the finite syzygy representation type of the category of finite length modules via the finite generation of the Grothendieck group of the category of finitely generated modules over $\widehat{R}$.
%The condition that the category of maximal Cohen--Macaulay modules over a ring $R$, denoted by $\CM(R)$, has finite representation type, is by definition that $R$ has finite Cohen--Macaulay representation type.  

% general, for $n> \dim R$, $\syz^n \fl(R)$ is a much smaller subcategory than $\syz \CM(R)$, however, we surprisingly note that for local CM rings of dimension $2$, the finite representation type of these two subcategories are equivalent.

%\begin{theorem}(see \ref{dim2}) Let $R$ be a local Cohen--Macaulay ring of dimension $2$ and with algebraically closed residue field. The following are equivalent. 
%\begin{enumerate}[\rm(1)]
%    \item $\syz \CM(R)$ has finite representation type.

%    \item $\add\syz^n  \fl(R)$ has finite representation type for some $n\geq 0$. 
%\end{enumerate}
    
%\end{theorem}

\begin{theorem}\label{C1}
(See \ref{application} and \ref{dim2}) Let $R$ be a Cohen--Macaulay local ring. Then:
\begin{enumerate}
    \item If $\add\Omega^n_R\CM_0(R)$ has finite representation type for some $n\geq 0$, then $\widehat{R}$ has an isolated singularity.

    \item If $\dim R=2$ and the residue field of $R$ is algebraically closed, then $\add\syz^n_R  \fl(R)$ has finite representation type for some $n\geq 0$ if and only if $\widehat R$ has isolated singularity and the Grothendieck group of $\mod(\widehat R)$ is a finitely generated Abelian group. 
\end{enumerate}
  
\end{theorem} 

In the above result, for each full subcategory $\X$ of $\mod(R)$, $\Omega^n_R \X$ (resp. $\add \X$) is the full subcategory of $\mod(R)$ whose objects are precisely the $n$-th syzygies (resp. direct summands) of modules in $\X$, and $\fl(R)$ denotes the full subcategory of $\mod(R)$ consisting of all finite length modules. In the case $n=1$, \Cref{C1} (1) shows that if $\Omega_R \CM_0(R)$ has finite representation type, then $\widehat R$ has an isolated singularity, and hence so does $R$.

Let $R$ be a Cohen--Macaulay local ring. Schreyer \cite{Sch} conjectured that $R$ has finite Cohen–Macaulay representation type if and only if $\widehat{R}$ does. This conjecture was confirmed by Leuschke and Wiegand \cite{LW2000} in the case where $R$ is excellent.
Theorem \ref{T2}, the main result of Section \ref{Section-ascent-descent}, provides a complete affirmative answer to Schreyer’s conjecture via its case $n = 0$ in part (2). The case $n = 0$ of
Theorem 1.3 (2), i.e., the original Schreyer's conjecture, was also independently
resolved by Toshinori Kobayashi, although he did not make it public; we thank Takahashi for bringing this to our attention. We also consider Schreyer-type questions for higher torsionfree modules. Theorem \ref{T2} answers the corresponding questions affirmatively. In the following, $\TF_n(R)$ is the category of $n$-torsionfree $R$-modules; c.f. \ref{def-torsionfree}.

%In the broader context of modern Cohen–Macaulay representation theory, interest has expanded from maximal Cohen–Macaulay modules to include their higher syzygies and 
%Building on the work of Kobayashi, Lyle, and Takahashi \cite{KLT}, we investigate a syzygy-type problem related to a result of theirs and obtain a generalization; see Proposition \ref{generalization-KLT}.

\begin{thm}\label{T2} (see \ref{cm category version})
Let $R$ be a Cohen--Macaulay local ring with Krull dimension $d$. For each $n\geq 0$, we have:

\begin{enumerate}
 \item $\add\Omega^n_R\CM_0 (R)$ has finite (resp. countable) representation type if and only if $\add\Omega^n_{\widehat{R}}\CM_0 (\widehat{R})$ has finite (resp. countable) representation type.
 
 \item $\add\Omega^n_R\CM(R)$ has finite representation type if and only if $\add\Omega^n_{\widehat{R}}\CM (\widehat{R})$ has finite representation type.

 \item Suppose that $n\in\{d, d+1\}$.
 Then $\TF_n(R)$ has finite representation type if and only if $\TF_n(\widehat{R})$ has finite representation type.
\end{enumerate}

\end{thm}

For an Artin algebra $A$, Chen \cite{Chen:2008} proved that if $A$ is Gorenstein, then the category of finitely generated left Gorenstein projective $A$-modules has finite representation type if and only if every left Gorenstein projective $A$-module is a direct sum of finitely generated ones. This result serves as a Gorenstein analogue of Auslander's celebrated theorem,
which asserts that for an Artin algebra $A$, the category of finitely generated left $A$-modules has finite representation type if and
only if every left $A$-module is a direct sum of finitely generated modules; see \cite{Auslander:1974, Auslander:1982}. In \cite{Beligiannis:adv}, Beligiannis observed that an Artin algebra $A$ satisfies the property that every left Gorenstein projective $A$-module is a direct sum of finitely generated ones if and only if $A$ is virtually Gorenstein and the category of finitely generated left Gorenstein projective $A$-modules has finite representation type.  In \cite{EH}, Eisenbud and Herzog gave a classification of homogeneous Cohen--Macaulay rings with finite Cohen--Macaulay type, and observed that such rings are stretched. 
%In addition, Chen \cite[Appendix C]{Chen-arXiv} raised the following questions for an Artin algebra $A$: 
%\begin{itemize}
 %   \item Is it true that $\Gproj(A)$ has finite representation type if and only if every left Gorenstein projective $A$-module is a direct sum of finitely generated ones?
 %    \item  If $\Gproj(A)$ has finite representation type, is $A$ necessarily virtually Gorenstein (c.f. \ref{defofvirtuallyG})?

  %  \item  If $\Gproj(A)=\proj(A)$, does it follow %that  $\GProj(A)=\Proj(A)$?
%\end{itemize}
%In the above, $\Proj(A)$ (resp. $\proj(A))$ is the category of left projective (resp. finitely generated projective) $A$-modules, $\GProj(A)$ (resp. $\Gproj(A))$ is the category of left Gorenstein projective (resp. finitely generated Gorenstein projective) $A$-modules. 

In Section \ref{Section-rings-F-CM type}, we investigate analogues of Chen’s questions \cite{Chen-arXiv} in the setting of Cohen–Macaulay rings of finite Cohen–Macaulay representation type; see \Cref{context-CM}. The main result of this section, Theorem \ref{T3}, establishes affirmative answers to these questions under the additional assumption that $R$ is complete.
%\begin{ques}
 %   Let $R$ be a Cohen--Macaulay local ring with finite Cohen--Macaulay representation type. 
%\begin{enumerate}[label=(\alph*)]
 %   \item  Is every left Gorenstein projective $A$-module necessarily a direct sum of finitely generated ones?

  %  \item  Is $R$ necessarily virtually Gorenstein?

 %   \item  If, in addition, $\Gproj(R)=\proj(R)$, does it follow that  $\GProj(R)=\Proj(R)?$
%\end{enumerate}
%\end{ques}

\begin{theorem}\label{T3} (See \ref{Main-finite CM type})
 Let $R$ be a Cohen--Macaulay local ring. Assume $\CM_0(R)$ has finite representation type. Then:
 \begin{enumerate}

 \item  Either $R$ is a hypersurface or $\GProj(R)=\Proj(R)$.

 \item Every Gorenstein projective $\widehat R$-module is a direct sum of finitely generated ones. 

 \end{enumerate}
%If any of the above equivalent conditions holds, then either $R$ is a hypersurface or $\GProj(R)=\Proj(R)$. 
\end{theorem}
In the above result, $\Proj(R)$ (resp. $\GProj(R)$) is the category of all projective (resp. Gorenstein projective) $R$-modules.
For a commutative Noetherian local ring $R$, if the category of finitely generated left Gorenstein projective $R$-modules, denoted by $\Gproj(R)$, has finite representation type, Christensen, Piepmeyer, Striuli, and Takahashi \cite{CPST} showed that either $R$ is Gorenstein or $\Gproj(R)=\proj(R)$; here, $\proj(R)$ is the category of finitely generated projective $R$-modules. In particular, when $R$ is Cohen--Macaulay and of finite Cohen--Macaulay representation type, it follows that either $R$ is Gorenstein or $\Gproj(R)=\proj(R)$; this application can also be deduced from Theorem \ref{T3} (1). 
Indeed, in general, if $\GProj(R) = \Proj(R)$, then $\Gproj(R) = \proj(R)$. Whether the converse holds, however, remains an open question; see Section \ref{Section-rings-F-CM type}.

More recently, Takahashi \cite{Takahashi:2023} introduced the notion of \emph{dominant local rings}, and discovered several remarkable properties of such rings—including a classification of thick subcategories in the bounded derived category of a local ring whose certain localizations are dominant. In Section \ref{Section-application}, we establish \Cref{application-dominant} concerning dominant local rings, providing a class of virtually Gorenstein rings that are not Gorenstein.  Another consequence of Theorem \ref{application-dominant} is the following: For a Cohen–Macaulay local ring $R$ of Krull dimension $2$ with an algebraically closed residue field, if $\Omega_R\CM_0(R)$ has finite representation type, then either $R$ is a hypersurface or $\GProj(R) = \Proj(R) $; see Proposition \ref{several applications}.
%\begin{thm}\label{T4} (See \ref{application-dominant})
%    Let $R$ be a dominant local ring. Assume $\TF_n(R)\cap \mod_0(R)$ has finite representation type for some $n\leq \depth(R)+1$. Then:
%    \begin{enumerate}
%        \item Either $R$ is a  hypersurface or $\GProj(R)=\Proj(R)$.

 %       \item  $\widehat R$ is virtually Grorenstein.
%    \end{enumerate}
%\end{thm}  

\begin{ack}
    We would like to thank Srikanth Iyengar and Ryo Takahashi for their many valuable discussions and comments. Souvik Dey was partially supported by the Charles University Research Center program No.UNCE/SCI/022 and a grant GA \v{C}R 23-05148S from the Czech Science Foundation.  Kaito Kimura was partly supported by Grant-in-Aid for JSPS Fellows 23KJ1117.  Jian Liu was supported by the National Natural Science Foundation of China (No. 12401046) and the Fundamental Research Funds for the Central Universities (Nos. CCNU24JC001, CCNU25JC025, CCNU25JCPT031). Yuya Otake was partly supported by Grant-in-Aid for JSPS Fellows 23KJ1119.
\end{ack}

\section{Notation and terminology}
Throughout this article, let $R$ be a commutative Noetherian ring. Denote by $\Mod(R)$ the category of $R$-modules, and by $\mod(R)$ the full subcategory of $\Mod(R)$ consisting of finitely generated $R$-modules. Let $\fl(R)$ denote the full subcategory of $\mod(R)$ consisting of modules of finite length. We write $\dim(R)$ for the Krull dimension of $R$, and $\depth(M)$ for the depth of an $R$-module $M$ for each $M \in \mod(R)$. For a local ring $(R, \mathfrak{m})$, the $\mathfrak{m}$-adic completion of an $R$-module $M$ is denoted by $\widehat{M}$. For a full subcategory $\X$ of $\mod(R)$, the \textit{additive closure} $\add\X$ is defined as the subcategory of $\mod(R)$ consisting of direct summands of finite direct sum of modules in $\X$.

\begin{chunk}
   \textbf{Spectrum, singular locus, and isolated singularity.}  
Let \( \Spec(R) \) denote the set of all prime ideals of $R$, equipped with the Zariski topology. The closed subsets in this topology are of the form
$$
\V (I) \colonequals \{\p \in \Spec(R) \mid I\subseteq \p\}
$$ for some ideal $I$ of $R$.

  The \emph{singular locus} $\Sing(R)$ of $R$ is the set of prime ideals of $R$ such that $R_\p$ is not regular; recall that a commutative Noetherian ring $R$ is regular if and only if each finitely generated $R$-module has finite projective dimension.
   We say that $R$ has an \textit{isolated singularity} if any non-maximal prime ideal of $R$ does not belong to $\Sing(R)$.  
\end{chunk}
\begin{chunk}
  \textbf{Annihilator and support of modules.}  
  Let $M$ be an $R$-module. We write $\ann_RM$ to be the \emph{annihialtor} of $M$ over $R$. That is, $\ann_RM=\{a\in R\mid a\cdot M=0\}$.  For simplicity, we will use $\ann M$ to represent $\ann_R M$ if there is no confusion. 
The \emph{support} of $M$ over $R$ is
$$\Supp_R M\colonequals\{\p\in\Spec(R)\mid M_\p\neq 0\}.$$
If $M$ is a finitely generated $R$-module, then $\Supp_RM=\V(\ann M)$. 
\end{chunk}
\begin{chunk}
      An $R$-module $M$ is said to be \textit{locally free on the punctured spectrum} if $M_\p$ is a free $R_\p$-module for any non-maximal prime ideal $\p$ of $R$.
Let $\mod_0(R)$ denote the full subcategory of $\mod(R)$ consisting of modules that are locally free on the punctured spectrum.
  % For $n\ge 0$ and $\Y\in\{\mod, \CM, \TF_n\}$, $(\Y R)_0$ is also written as $\Y_0 R$.
\end{chunk}
   \begin{chunk}\label{MCM}
  \textbf{Maximal Cohen--Macaulay modules.} A finitely generated $R$-module $M$ is called \emph{maximal Cohen--Macaulay} if $\depth(M_\p)\geq \dim(R_\p)$ for each prime ideals $\p$ of $R$; see \cite{BH, LW, Yoshino} for more details.
   The full subcategory of $\mod(R)$ consisting of maximal Cohen--Macaulay $R$-modules is denoted by $\CM(R)$. A commutative Noetherian ring $R$ is said to be \emph{Cohen--Macaulay} if $R\in\CM(R)$.

      Let $\CM_0(R)$ denote the full subcategory of $\CM(R)$ consisting of modules that are locally free on the punctured spectrum. That is, $\CM_0(R)\colonequals\CM(R)\cap \mod_0(R)$.
   \end{chunk}
\begin{chunk}\label{def of res}
\textbf{Resolving subcategories.}
Let $\X$ be a full subcategory of $\mod(R)$. The subcategory $\X$ is called \emph{resolving} if it contains $\proj(R)$, is closed under direct summands, and it satisfies the following conditions: for any short exact sequence $0\rightarrow M_1\rightarrow M_2\rightarrow M_3\rightarrow 0$ in $\mod(R)$, 
\begin{enumerate}
    \item If $M_1,M_3\in\X$, then $M_2\in\X$;
    \item If $M_2,M_3\in\X$, then $M_1\in\X$. 
\end{enumerate}
For example, if $R$ is Cohen--Macaulay, then $\CM(R)$ is a resolving subcategory of $\mod(R)$.
\end{chunk}
 \begin{chunk}
   \textbf{Finite (resp. countable) representation type}.
   Let $\X$ be a subcategory of $\mod(R)$. 
   We say that $\X$ has \emph{finite (resp. countable) representation type} if there are only finitely (resp. countably) many isomorphism classes of indecomposable modules belonging to $\X$.

It is also common to say that $R$ has finite Cohen--Macaulay type when $\CM(R)$ has finite representation type.
  \end{chunk}

\begin{chunk}\label{def of syzygy}
\textbf{Syzygy modules.} 
For a finitely generated $R$-module $M$, let $\Omega_RM$ denote the \emph{first syzygy} of $M$. That is, there is a short exact sequence
$$0\rightarrow \Omega_R M\rightarrow P\rightarrow  M\to 0,$$
where $P$ is a finitely generated projective $R$-module.
By Schanuel’s Lemma, $\Omega_RM$
is independent of the choice of the projective resolution of M up to projective summands.

   We put $\Omega^0_RM=M$, and the {\em $n$-th syzygy} $\Omega^n_R M$ is defined inductively by $\Omega^n_R M=\Omega_R(\Omega^{n-1}_RM)$ for each $n>0$. If there is no confusion, we will use $\Omega^nM$ to represent $\Omega^n_R M$. 
For a subcategory $\X$ of $\mod(R)$, we denote
$$\Omega^n\X \colonequals \{M\in\mod(R)\mid M\cong \Omega^nX \text{ for some } X\in \X\}.$$
For a local ring $R$ and a full subcategory $\Y$ of $\mod(\widehat R)$, we will use $\Omega^n\Y$ to represent $\Omega^n_{\widehat R}\Y$ in this article. 
\end{chunk}

  \begin{chunk}\label{def-torsionfree}
      \textbf{Auslander transpose and $n$-torsionfree modules.} Assume $R$ is a local ring and $M$ is a finitely generated $R$-module. Let $$F=\cdots \rightarrow  F_1\xrightarrow{d_1}F_0\rightarrow  M\to 0$$ 
be a minimal free resolution of $M$.  Denote by $(-)^\ast$ the $R$-dual $\Hom_R(-,R)$. The \emph{(Auslander) transpose} of $M$, denoted by $\Tr_R M$, is defined to be the cokernal of $(d_1)^\ast$; see \cite{AB} for more details. Note that $\Omega^2\Tr_RM\cong M^\ast$ and there exists an exact sequence
$$
0\to \Ext^1_R(\Tr_RM,R)\to M\xrightarrow{\varphi_M}M^{\ast\ast}\to\Ext^2_R(\Tr_RM,R)\to0,
$$
where $\varphi_M:M\to M^{\ast\ast}$ is the canonical homomorphism given by $\varphi_M(x)(f)=f(x)$ for $x\in M$ and $f\in M^\ast$.
Moreover, the image of $\varphi_M$ is isomorphic to $\Omega_R\Tr_R\Omega_R\Tr_R M$ up to projective summands; see \cite[Appendix]{AB}.

  For each $n\geq 0$,  $M$ is called \emph{$n$-torsionfree} if $\Ext_R^i(\Tr_R M,R)=0$ for all $1\le i\le n$. For each $M\in\mod(R)$, if $M$ is $n$-torsionfree, then $M$ is a $n$-th syzygy module; see \cite[Theorem 2.17]{AB}.
 The full subcategory of $\mod(R)$ consisting of $n$-torsionfree $R$-modules is denoted by $\TF_n(R)$. We denote $\TF_n^0(R)\colonequals \TF_n(R)\cap \mod_0(R)$.
  \end{chunk}
\begin{chunk}
  \textbf{Gorenstein rings and hypersurfaces.}  For a commutative Noetherian ring $R$, it is said to be \emph{Gorenstein} provided that $R_\p$ has finite injective dimension as an $R_\p$-module for each prime ideal $\p$ of $R$. 
For example, every regular ring is Gorenstein. 

For a commutative Noetherian local ring $R$, it is said to be a \emph{hypersurface} if its completion $\widehat R$ is isomorphic to a quotient of a regular local ring by a principal ideal. Every hypersurface is Gorenstein. 
\end{chunk}
\begin{chunk}
\textbf{Gorenstein projective (resp. injective) modules.}
The notions of Gorenstein projective and Gorenstein injective modules will be used in Section \ref{Section-rings-F-CM type} and Section \ref{Section-application}. For details, see \cite[Chapter 10]{Enochs-Jenda}.

Let $\Proj(R)$ (resp. $\Inj(R)$) denote the category of projective (resp. injective) $R$-modules.
An $R$-module $M$ is called \emph{Gorenstein projective} (resp. \emph{Gorenstein injective}) if there exists an acyclic complex of projective (resp. injective) $R$-modules
$$
\mathbf{P} : \cdots \longrightarrow P_1 \xrightarrow{d_1} P_0 \xrightarrow{d_0} P_{-1} \longrightarrow \cdots \quad \text{(resp. } \mathbf{I} : \cdots \longrightarrow I_1 \xrightarrow{\partial_1} I_0 \xrightarrow{\partial_0} I_{-1} \longrightarrow \cdots\text{)}
$$
such that $\Hom_R(\mathbf{P}, Q)$ (resp. $\Hom_R(E, \mathbf{I})$) remains acyclic for every $Q \in \Proj(R)$ (resp. $E \in \Inj(R)$), and $M$ is isomorphic to the image of $d_0$ (resp.  $\partial_0$). For example, any projective (resp. injective)  $R$-module is Gorenstein projective (resp. Gorenstein injective). 
We write $\GProj(R)$ (resp. $\GInj(R)$) as the full subcategory of $\Mod(R)$ consisting of Gorenstein projective (resp. Gorenstein injective) modules.  

Let $\Gproj(R)$ denote the full subcategory of $\mod(R)$ consisting of finitely generated Gorenstein projective $R$-modules. If $R$ is Gorenstein, then $\Gproj(R)=\CM(R)$; see \cite[Proposition 3.8]{AB} and \cite[3.1.24]{BH}. 
\end{chunk}

\begin{chunk}\label{defofvirtuallyG}
\textbf{Virtually Goresntein rings.}
In \cite{Beligiannis-Reiten}, Beligiannis and Reiten introduced the notion of virtually Gorenstein Artin algebras. This concept was extended to commutative Noetherian rings of finite Krull dimension by Zareh-Khoshchehreh, Asgharzadeh, and Divaani-Aazar \cite{ZAD}. Recently, Di, Liang, and Wang \cite{DLW} further generalized the notion of virtually Gorenstein rings to the setting of arbitrary rings.

Precisely, a commutative Noetherian ring $R$ is said to be \emph{virtually Gorenstein} if $ \GProj(R)^\perp=^\perp \GInj(R)$, where $\GProj(R)^\perp\colonequals\{Y\in\Mod(R)\mid \Ext^i_R(X,Y)=0\text{ for all }i>0 \text{ and }X\in\GProj(R)\}$ and $^\perp\GInj(R)\colonequals\{X\in\Mod(R)\mid \Ext^i_R(X,Y)=0\text{ for all }i>0 \text{ and }Y\in\GInj(R)\}$. It is known that any Gorenstein ring is virtually Gorenstein; see \cite[Example 3.13 (i)]{ZAD} for finite Krull dimension case and \cite[Theorem A.1]{Iyengar-Krause2022} for the general case. 
\end{chunk}

\begin{chunk}
    \textbf{Thick subcategories.} Let $\T$ be a triangulated category. A full subcategory $\mathcal C$ of $\T$ is called \emph{thick} provided that it is closed under suspensions, cones, and direct summands.
For a subcategory $\mathcal C$ of $\T$. Let $\thick_\T(\mathcal C)$ denote the smallest thick subcategory of $\T$ containing $\mathcal C$.
\end{chunk}
\begin{chunk}
    \textbf{Derived categories.} 
    Let $\D(R)$ denote the \emph{derived category} of $R$-modules. It is a triangulated category equipped with a suspension functor $[1]$, where for each complex $X$, the shifted complex $X[1]$ is given by $(X[1])_i = X_{i-1}$ and $\partial_{X[1]} = -\partial_X$. The \emph{bounded derived category}, denoted by $\D^f(R)$, is the full subcategory of $\D(R)$ consisting of complexes $X$ whose total homology $\bigoplus_{i\in\Z} \h_i(X)$ is finitely generated over $R$. Note that $\D^f(R)$ inherits the triangulated structure from $\D(R)$ and forms a thick subcategory of $\D(R)$.
\end{chunk}

%\begin{chunk}
 %    Let $\mathcal X$ be a full subcategory of  $\mod(R)$, following \cite[3.4]{DT2015}, the \emph{radius of $\mathcal X$} is defined to be
%     \begin{center}
%         $\radius(\mathcal X)\colonequals\inf\{n\in \mathbb N\mid \mathcal X\subseteq [G]_n \text{ for some } G\in \mathcal X\}$.
%     \end{center}
 %   This definition of the radius is different from that in \cite[Definition 2.3]{DT2014}. The definition of the radius in \cite{DT2014} does not require $G\in \mathcal X$. 
%\end{chunk}
\section{Finite syzygy representation type}
In this section, we investigate finite syzygy representation type. The main results are Theorems \ref{equivalent conditions} and \ref{main}.  In the following, the notation $\Omega^n\fl(R)$ (resp. $\Omega^n\fl(\widehat R)$) will represent $\Omega^n_R\fl(R)$ (resp. $\Omega^n_{\widehat R}\fl(\widehat R))$ for each $n\geq 0$; see \ref{def of syzygy}. 
For subcategories $\X,\Y$ of $\mod(R)$, set $\Ext_R(\X,\Y)\colonequals \bigoplus_{i>0,X\in\X,Y\in\Y}\Ext^i_R(X,Y).$
\begin{lem}\label{contain}
       Let $R$ be a commutative Noetherian local ring and $\X$ be a subcategory of $\mod(R)$. If $\mathcal X$ contains $\Omega^n\fl({R})$ for some $n\geq 0$, then, for each $m\geq 0$, 
       $$\Sing(R)\subseteq  \V(\ann\Ext_R(\mathcal{X}, \Omega^m\fl(R))) \text{ and } \Sing(\widehat R)\subseteq  \V(\widehat{R}\otimes_R \ann\Ext_R(\mathcal{X}, \Omega^m\fl(R))). $$
\end{lem}
\begin{proof}
   Set $I=\ann\Ext(\Omega^n\fl(R),\Omega^m\fl(R))$. 
 By \cite[Proposition 3.6]{Mifune2024}, $\Sing(R)\subseteq \V(I)$. Since $\Omega^n\fl(R)\subseteq\X$, we have
 $\ann\Ext_R(\mathcal{X}, \Omega^m\fl(R))\subseteq I$. Hence,
 $$\Sing(R)\subseteq \V(I)\subseteq \V(\ann\Ext_R(\mathcal{X}, \Omega^m\fl(R))).$$
 
 Note that there is an equivalence
    $
    \widehat{R}\otimes_R-\colon \fl(R)\xrightarrow \cong \fl(\widehat{R}).
    $
    This implies that $I\subseteq \ann_{\widehat{R}}\Ext_{\widehat R}(\Omega^n\fl(\widehat{R}),\Omega^m\fl(\widehat{R})$. In particular, $\widehat I\subseteq \ann_{\widehat{R}}\Ext_{\widehat R}(\Omega^n\fl(\widehat{R}),\Omega^m\fl(\widehat{R}))$. It follows that 
    $$
    \widehat{R}\otimes_R \ann\Ext_R(\mathcal{X}, \Omega^m\fl(R))\subseteq \widehat{R}\otimes_R I=\widehat{I}\subseteq \ann_{\widehat{R}}\Ext_{\widehat R}(\Omega^n\fl(\widehat{R}),\Omega^m_{\widehat{R}}\fl(\widehat{R})).
    $$
    This yields the second inclusion below:
    $$
\Sing(\widehat R)\subseteq  \V(\ann_{\widehat{R}}\Ext_{\widehat R}(\Omega^n\fl(\widehat{R}),\Omega^m\fl(\widehat{R})))\subseteq \V(\widehat{R}\otimes_R \ann\Ext_{R}(\X,\Omega^m\fl({R}))),
$$
where the first inclusion is by \cite[Proposition 3.6]{Mifune2024}.
%By Lemma \ref{contain}, $\Sing(R)\subseteq \V(\ann\Ext_R(\Omega^n(\fl(R),\fl(R)))$. Combining this with $\Omega^n\fl(R)\subseteq \mathcal X$, we conclude that $\Sing(R)\subseteq \V(\ann\Ext_R(\mathcal{X}, \fl(R)))$.
\end{proof}
\begin{chunk}
    Let $\X$ be a subcategory of $\mod(R)$. Set $$\NF(\X)\colonequals\{\p\in\Spec(R)\mid X_\p \text{ is not free for some }X\in\X\}$$ to be the \emph{nonfree locus} of $\X$. If $\X=\{X\}$ consists of a single module $X\in\mod(R)$, we write $\NF(X)$ in place of $\NF(\{X\})$. In this case, $\NF(X)=\Supp_R \Ext^1_R(X,\Omega X)$.
\end{chunk}

\begin{chunk}\label{composition}
Let $\X$ and $\Y$ be subcategories of $\mod(R)$. Denote by $\X\circ\Y$ the subcategory of $\mod(R)$ consisting of modules $Z$ which fit into a short exact sequence $0\to X\to Z\to Y\to 0$ with $X\in\X$ and $Y\in\Y$. In \cite[Definitions 2.1 and 5.1]{DT2014}, 
Dao and Takahashi introduced two constructions based on this operation:

(1)  An
ascending chain of full subcategories built out of $\X$ as follows:
   For $r=1$, $|\X|_1\colonequals |\X|$, where $|\X|\colonequals \add\X$; For $r>1$,  $|\X|_r\colonequals ||\X|_{r-1}\circ \X|$. For each $i,j\geq 1$, $||\mathcal X|_i\circ |\mathcal X|_j|=|\mathcal X|_{i+j}$, and hence $||\mathcal X|_i|_j=|\mathcal X|_{ij}$.

  (2)  The \emph{ball of radius $r$ centered at $\mathcal X$} as follows:
   For $r=1$, $[\mathcal{X}]_1\colonequals [\mathcal{X}]$, where $[\mathcal X]\colonequals\add\{R\cup \bigcup_{i\geq 0}\Omega^i\X\}$;
 For $r> 1$, $[\mathcal X]_r\colonequals[[\mathcal X]_{r-1}\circ [\mathcal X]]$. 
Similarly, this satisfies $[[\mathcal X]_i\circ [\mathcal X]_j]=[\mathcal X]_{i+j}$ for each $i,j\geq 1$, and hence $[[\mathcal X]_i]_j=[\mathcal X]_{ij}$.

In \cite[Definition 3.3]{DT2015}, Dao and Takahashi introduced the \emph{dimension} of $\X$, denoted by
$\dim \X$, to be the infimum of the integers $n \geq 0$ such that $\X = [G]_{n+1}$ for some
$G \in  \X$.
\end{chunk}
\begin{lem}\label{radius finite}
    Let $\mathcal X$ be a subcategory of $\mod(R)$.
    Then:
\begin{enumerate}

    \item $\NF(\mathcal X)\subseteq \V(\ann\Ext_R(\mathcal{X}, \mod(R))).$

    \item  If there exists $G\in \mod(R)$ and $i\geq 0$ such that $\mathcal X\subseteq [G]_i$, then $$\V(\ann\Ext_R(\mathcal{X}, \mod(R)))\subseteq \NF(G).$$
\end{enumerate}

\end{lem}
\begin{proof}

    (1) The statement follows from the equality $\NF(\X)=\Supp_R\Ext_R(\X,\mod(R))$; see \cite[Remark 5.2]{DT2015}.
Next, we prove (2). Since $\mathcal X\subseteq [G]_i$, we have
    $$ \V(\ann\Ext_R(\mathcal{X}, \mod(R)))\subseteq\V(\ann\Ext_R([G]_i, \mod(R))).$$ It follows from \cite[Lemma 5.3]{DT2015} that $\V(\ann\Ext_R([G]_i, \mod(R)))=\V(\ann\Ext_R(G, \mod(R)))$.  Note that $\ann\Ext_R(G,\mod(R))=\ann\Ext^1_R(G,\Omega G)$; see \cite[Lemma 2.14]{IT2016}. Thus, we conclude that
 $$
  \V(\ann\Ext_R(\mathcal{X}, \mod(R)))\subseteq\V(\ann \Ext^1_R(G,\Omega G))=\NF(G).
 $$
 This completes the proof.
\end{proof}
 \begin{chunk}\label{def-of-ca}
     Let $R$ be a commutative Noetherian ring. For each integer $n \ge 0$, denote by $\ca^n(R)$ the ideal consisting of elements $a \in R$ such that $a \cdot \Ext_R^n(M,N) = 0$ for all finitely generated $R$-modules $M$ and $N$.
   The {\em cohomology annihilator}, introduced by Iyengar and Takahashi \cite[Definition 2.1]{IT2016},  is defined as the union
$$
\ca(R) := \bigcup_{n \ge 0} \ca^n(R).
$$
Since $R$ is Noetherian, the ascending chain of ideals
$
\ca^0(R) \subseteq \ca^1(R) \subseteq \ca^2(R) \subseteq \cdots
$
stabilizes, and hence $\ca(R) = \ca^n(R)$ for some $n \ge 0$.
 \end{chunk}
\begin{prop}\label{Singular locus equal}
  Let $R$ be a quasi-excellent local ring. Then
  $$
  \Sing(R)=\V(\ann\Ext_R(\mathcal{X},\mathcal{Y}))
  $$
  for each $\Omega^n\fl(R)\subseteq\mathcal{X}\subseteq \Omega^m\mod(R)$  and each  $\Omega^l\fl(R)\subseteq \mathcal Y\subseteq \mod(R)$ with $n,l\geq 0$ and $m\geq \dim(R)$. 
\end{prop}
\begin{proof}
  Set $d=\dim(R)$. By Lemma \ref{contain},  one has
$\Sing(R)\subseteq \V(\ann\Ext_R(\Omega^n\fl(R),\Omega^l\fl(R))).$ Combining with $\Omega^n\fl(R)\subseteq \mathcal X$ and $\Omega^l\fl(R)\subseteq \mathcal Y$, we get that $\Sing(R)\subseteq \V(\ann\Ext_R(\mathcal X,\mathcal Y))$.
   
It remains to prove $\V(\ann\Ext_R(\mathcal X,\mathcal{Y}))\subseteq \Sing(R)$.
Note that $\ca^{d+1}(R)\subseteq \ann\Ext_R(\mathcal{X},\mathcal{Y})$ as $\mathcal X\subseteq \Omega^m\mod(R)$ with $m\geq d$. Hence, 
$$
\V(\ann\Ext_R(\mathcal{X},\mathcal{Y}))\subseteq \V(\ca^{d+1}(R)).
$$
Since $R$ is quasi-excellent, \cite[Theorem 1.1]{Kimura} implies that $\ca^{d+1}(R)$ defines the singular locus of $R$. That is,
$
\Sing(R)=\V(\ca^{d+1}(R)).
$
This completes the proof.
\end{proof}

In \cite[Theorem 1.1 (1)]{DT2015}, Dao and Takahashi observed that, for an equicharacteristic complete Cohen--Macaulay local ring $R$ with a perfect residue field, the ring $R$ has an isolated singularity if and only if the dimension of $\CM_0(R)$ is finite, that is, there exists an object $G\in\CM_0(R)$ and an integer $r\ge0$ such that $\CM_0R=[G]_r$.

When $R$ is not necessarily Cohen--Macaulay, we may not know whether $\CM_0(R)$ contains any nontrivial objects, so we need to consider a different subcategory, as we do below. It is worth mentioning that our study applies to arbitrary local rings, without imposing any additional assumptions.
\begin{theorem}\label{equivalent conditions}
   Let $(R,\m,k)$ be a commutative Noetherian local ring with depth $t$ and Krull dimension $d$.  For a subcategory $\X$ of $\mod(R)$, consider the following conditions:
\begin{enumerate}
\item $\X\subseteq |\bigoplus_{i=0}^d\Omega^ik|_r$ for some $r\geq 1$.

 \item There exists $G\in\mod_0(R)$such that $\mathcal X\subseteq |G|_r$ for some $r\geq 1$.
 
    \item There exists $G\in\mod_0(R)$ such that $\mathcal X\subseteq [G]_r$ for some $r\geq 1$.

\item For each subcategory $\mathcal Y$ of $\mod(R)$, the ideal $\ann\Ext_R(\mathcal X,\mathcal{Y})$ contains some power of the maximal ideal.

    \item There exists a subcategory $\mathcal Y$ of $\mod(R)$ containing $\Omega\X$ such that the ideal $\ann\Ext_R(\mathcal X,\mathcal{Y})$ contains some power of the maximal ideal. 

\item $\ca(R)$ contains some power of the maximal ideal. 

    \item  $\widehat R$ has an isolated singularity.

  %  \item $R$ has an isolated singularity.
\end{enumerate}
Then $(1)\iff (2)\iff (3)\iff (4)\iff (5)$ and $(6)\iff (7)$. The implication $(5)\Longrightarrow (6)$ holds if, in addition, $\mathcal X$ contains $\Omega^n\fl(R)$ for some $n\geq 0$.

All the above conditions are equivalent if, in addition, $\Omega^n\fl(R)\subseteq\mathcal{X}\subseteq \Omega^m\mod(R)$ for some $n\geq 0$ and $m\geq d$. In this case, they are also equivalent to the following condition:

~~~(8) $\X\subseteq |\bigoplus_{i=t}^d\Omega^ik|_r$ for some $r\geq 1$.
\end{theorem}
%\begin{theorem}\label{equivalent conditions}
%   Let $(R,\m,k)$ be a commutative Noetherian local ring and $\mathcal X$ be a full subcategory of $\mod(R)$. Consider the following conditions:
%\begin{enumerate}
%    \item There exists $G\in\mod_0(R)$ and $i\geq 0$ such that $\mathcal X\subseteq [G]_i$.

%\item For each subcategory $\mathcal Y$ of $\mod(R)$, the ideal $\ann\Ext_R(\mathcal X,\mathcal{Y})$ contains some power of the maximal ideal.

%    \item For some subcategory $\mathcal Y$ of $\mod(R)$ containing $\Omega\X$, the ideal $\ann\Ext_R(\mathcal X,\mathcal{Y})$ contains some power of the maximal ideal. 

 %   \item  $\widehat R$ has an isolated singularity.

 %   \item $R$ has an isolated singularity.
%\end{enumerate}
%Then $(1)\iff (2)\iff (3)$ and $(4)\Longrightarrow (5)$. The implication $(3)\Longrightarrow (4)$ holds if, in addition, $\mathcal X$ contains $\Omega^n(\fl(R)$ for some $n\geq 0$. All these conditions are equivalent if, in addition, $\Omega^n\fl(R)\subseteq\mathcal{X}\subseteq \Omega^t(\mod(R))$ for some $n\geq 0$ and $t\geq \dim(R)$, and $R$ is an equidimensional equicharacteristic complete local ring with a perfect residue field.
%\end{theorem}
\begin{proof}
%The implication $(4)\Longrightarrow (5)$ is known and can be obtained by \cite[Lemma 7.9]{LW}.
The implications $(1)\Longrightarrow(2)\Longrightarrow (3)$ and $(4)\Longrightarrow (5)$  are trivial.  The equivalence $(6)\iff (7)$ follows from \cite[Proposition 2.4 (1)]{Kimura-compact}.

$(3)\Longrightarrow (4)$. By Lemma \ref{radius finite} (2),  $\V(\ann\Ext_R(\mathcal{X}, \mod(R)))\subseteq \NF(G).$ Note that $\NF(G)\subseteq \{\m\}$ as $G\in \mod_0(R)$. Thus, $\V(\ann\Ext_R(\mathcal X,\mod(R)))\subseteq\{\m\}$. This yields that $\ann\Ext_R(\mathcal X,\mod(R))$ contains some power of the maximal ideal. The desired result follows as $\ann\Ext_R(\mathcal X,\mod(R))\subseteq \ann\Ext_R(\mathcal X,\mathcal Y)$.

$(5)\Longrightarrow (1)$. This can be proved by the same argument as in the proof of \cite[Theorem 3.1 (1)]{BHST}. Assume $\m^s\subseteq \ann\Ext_R(\mathcal X, \mathcal Y)$ for some $s>0$. Choose a system of parameter $\x=x_1,\ldots, x_d\subseteq \m^s$ of $R$. 
For each $M\in\X$, let $H_i$ be the $i$-th Koszul homology of $M$ with respect to $\x$ for each integer $i$.  Note that $\Omega M\in \mathcal Y$ as $\Omega\X\subseteq \mathcal Y$ by the assumption. Hence, $(\x)\cdot \Ext_R^1(M,\Omega M)=0$. Combining this with \cite[Lemma 2.14]{IT2016}, we get that $(\x)\cdot \Ext^i_R(M,\mod(R))=0$ for all $i\geq 1$. It follows from \cite[Corollary 3.2 (2)]{Takahashi:2014} that for each $1\le i\le d$, there is an short exact sequence 
$$0\rightarrow  H_i\rightarrow E_i\rightarrow  \Omega E_{i-1}\rightarrow  0$$ 
with $E_0=H_0$ such that $M$ is a dirct summand of $E_d$.
Thus, $M$ belongs to $|\bigoplus_{i=0}^d \Omega^i H_{d-i}|_{d+1}$.
On the other hand,  note that $R/(\x)$ is Artinian. Let $l$ denote the Loewy length of $R/(\x)$. That is, $l=\inf\{i\in \mathbb N\mid (\m/(\x))^i=0\}$. For each $0\leq i\leq d$, since $\x$ annihilates $H_i$, we have $\m^l\cdot \h_i=0$.
This implies that, for each  $0\leq i\leq d$,  $H_{d-i}$ is in $|k|_l$, and hence $\Omega^i H_{d-i}$ is in $|\Omega^ik|_l$.
Therefore, we obtain
$$M\in|\bigoplus_{i=0}^d \Omega^i H_{d-i}|_{d+1}\subseteq |\bigoplus_{i=0}^d \Omega^i k|_{l(d+1)}.$$
 This yields $(1)$. 

Assume, in addition, $\mathcal X$ contains $\Omega^n\fl(R)$ for some $n\geq 0$. Choose $\mathcal Y=\fl(R)$. Assume $(4)$ holds. This yields that $\widehat{R}\otimes_R \ann\Ext_R(\mathcal X,\mathcal Y)$ contains some powers of $\widehat \m$. By Lemma \ref{contain},  $(7)$ holds. This yields that $(4)\Longrightarrow (7)$, and hence $(5)\Longrightarrow (6)$ as $(4)\iff (5)$ and $(6)\iff (7)$.

Next, assume, in addition, $\Omega^n\fl(R)\subseteq\mathcal{X}\subseteq \Omega^m\mod(R)$ for some $n\geq 0$ and $m\geq d$.

 $(6)\Longrightarrow (5)$. Since $\X\subseteq \Omega^m\mod(R)$, we have $\ca^{m+1}(R)\subseteq \ann\Ext_R(\X,\Omega \X)$. As $d\leq m$, it follows that $\ca^{d+1}(R)\subseteq \ca^{m+1}(R)$, and hence $\ca^{d+1}(R)\subseteq \ann\Ext_R(\X,\Omega \X)$. By \cite[Theorem 1.1]{Kimura}, $$\V(\ca(R))=\V(\ca^{d+1}(R)).$$ 
Hence, the assumption in $(6)$ implies that $\ca^{d+1}(R)$ contains some power of the maximal ideal. It follows that $\ann \Ext_R(\X, \Omega \X)$ also contains some power of the maximal ideal. Setting $\mathcal{Y} = \Omega \X$ completes the proof.

It remains to show that all the conditions are equivalent to $(8)$. The implication $(8)\Longrightarrow (1)$ holds trivially. To prove $(1)\Longrightarrow (8)$, it suffices to prove $(5)\Longrightarrow (8)$ as $(1)\iff (5)$. This follows from a similar argument as in the proof of $(5)\Longrightarrow (1)$. Assume that condition $(5)$ holds.  For each $M\in\X$, let $\x$, $\h_i$, $E_i$, and $l$ be as in the proof of $(5)\Longrightarrow (1)$. Since $\X\subseteq \Omega^m\mod(R)$, we have $M\in\Omega^m\mod(R)$. Combining this with $m\geq d\geq t$, it follows from \cite[1.3.7]{BH} that $\depth(M)\geq t$. By \cite[Theorem 1.6.17]{BH}, this implies that $H_i=0$ for all $i>d-t$. It follows from the short exact sequences in the proof $(5)\Longrightarrow (1)$ that $M\in |\bigoplus_{i=t}^d \Omega^i H_{d-i}|_{d-t+1}$. Combining this with that $\Omega^i H_{d-i}$ is in $|\Omega^ik|_l$ for any $i$, we conclude that $M\in |\bigoplus_{i=t}^d \Omega^i k|_{l(d-t+1)}$. This completes the proof.
\end{proof}
\begin{rem}
(1) As mentioned in the proof of Theorem \ref{equivalent conditions}, the equivalence $(6)\iff (7)$ is due to the second author \cite[Proposition 2.4 (1)]{Kimura-compact}. 

(2) Keep the assumption as Theorem \ref{equivalent conditions} and assume $\Omega^n\fl(R)\subseteq\X\subseteq \Omega^m\mod(R)$ for some $n\geq 0$ and $m\geq d$. If any of the conditions $(1)-(8)$ holds, then, in statements $(1)$, $(2)$, and $(3)$ of Theorem \ref{equivalent conditions}, we can take $r=l(d-t+1)$, where $l=\inf\{\ell\ell_R(R/\x)\mid \x=x_1,\ldots,x_d \subseteq \ca^{m+1}(R) \text{ is a system of parameter of }R \}$; $\ell\ell_R(R/(\x))$ is the Loewy length of $R/(\x)$.

Indeed, the validity of any of the conditions $(1)$–$(8)$ implies that $\ca(R)$ contains some power of the maximal ideal. Then, by the proof of $(6) \Longrightarrow (5)$, it follows that $\ca^{m+1}(R)$ also contains some power of the maximal ideal. Since $\ca^{m+1}(R) \subseteq \ann \Ext_R(\X, \Omega \X)$, the same argument as in the proof of $(1) \Longrightarrow (8)$ shows that we can take $r = l(d - t + 1)$.
\end{rem}
Theorem \ref{equivalent conditions} is a generalization a result of Bahlekeh, Hakimian,  Salarian, and Takahashi \cite[Theorem 3.2]{BHST}. In fact, the equivalence $(3) \iff (4)\iff (5)$ below was established in loc. cit.
\begin{cor}\label{extends-BHST}
    Let $(R,\m,k)$ be a commutative Noetherian local ring with depth $t$ and Krull dimension $d$. The following are equivalent.
    \begin{enumerate}
        \item For each $n\geq d$, there exists $r\geq 1$ such that $\Omega^n\mod(R)\subseteq |\bigoplus_{i=t}^d\Omega^ik|_r$.

        \item $\Omega^n\mod(R)\subseteq |\bigoplus_{i=t}^d\Omega^ik|_r$ for some $n\geq d$ and $r\geq 1$.

         \item $\Omega^n\mod(R)\subseteq |\bigoplus_{i=0}^d\Omega^ik|_r$ for some $n\geq 0$ and $r\geq 1$.

 \item $R$ has an isolated singularity and there exists $G\in\mod(R)$ such that $\Omega^n\mod(R)\subseteq [G]_r$ for some $n\geq 0$ and $r\geq 1$.
 
         \item $\ca(R)$ contains some power of the maximal ideal. 
    \end{enumerate}
\end{cor}
\begin{proof}
    The implications $(1)\Longrightarrow (2)\Longrightarrow (3)$ are trivial. 

    $(3)\Longrightarrow (4)$. By Theorem \ref{equivalent conditions}, $\widehat R$ has an isolated singularity, and hence so does $R$. The second statement of (4) holds as $|\bigoplus_{i=0}^d\Omega^ik|_r\subseteq [\bigoplus_{i=0}^d\Omega^ik]_r$. 

$(4)\Longrightarrow (5)$. Assume $R$ has an isolated singularity and $\Omega^n\mod(R)\subseteq [G]_r$. It follows from this that $\Omega^{n+d}\mod(R)\subseteq [\Omega^d G]_r$. Since $R$ has an isolated singularity, $\Omega^dG\in\mod_0(R)$. Choose $\X=\Omega^{n+d}\mod(R)$. By Theorem \ref{equivalent conditions}, $\ca(R)$ is contains some power of the maximal ideal. 

   $(5)\Longrightarrow (1)$.  Let $\X = \Omega^n\mod(R)$. Since $\ca(R)$ contains some power of the maximal ideal and $n \geq d$, the implication follows from $(6)\Rightarrow (8)$ in Theorem \ref{equivalent conditions}
\end{proof}
\begin{rem}
 For a commutative Noetherian local ring $(R,\m,k)$, if there exists $n,r\geq 0$ and $G\in\mod(R)$ such that $\Omega^n(R/\m^i)\subseteq [G]_r$ for all $i\geq 0$, then $n\geq \dim(R)$; see \cite[4.6]{DLMO}. Since $|G|_r\subseteq [G]_r$, the condition of $(3)$ in Corollary \ref{extends-BHST} always yields that $n\geq \dim(R)$.
\end{rem}

\begin{lem}\label{add} Let $\X$ be a full subcategory of $\mod (R)$ that is closed under finite direct sums and direct summands. If $\X$ has finite representation type, then $ \X=\add M$ for some $M\in \X$.  The converse holds if, in addition, $R$ is semilocal.
\end{lem}

\begin{proof} Assume $\X$ has finitely many indecomposable modules. Let $M_1, M_2, \ldots, M_n $ be all representatives of isomorphism classes of indecomposable modules belonging to $\X$. Let $M:=M_1\oplus \cdots \oplus M_n$.  Then $M\in \X$ and $M_i\in \add M \subseteq  \X$. Let $X\in  \X$, and let $X=X_1\oplus \cdots \oplus X_r$ be an indecomposable decomposition of $X$.  Since $ \X$ is closed under direct summands, for each $1\leq i\leq r$, there exists $1\leq j\leq n$ such that $X_i\cong M_j$, and hence $X\in \add M$.  This implies that $\X=\add M$.

Assume $\X=\add M$ for some $M\in\X$ and $R$ is semilocal. By \cite[Theorem 1.1]{Wiegand}, $\X$ has finitely many indecomposable modules. 
\end{proof}  

%\begin{cor}\label{res} Let $\X$ be a resolving subcategory of $\mod (R)$. If $\Omega (\X \cap \mod_0(R))$ has only finitely many indecomposable objects, then  $ \Omega (\X\cap \mod_0(R))=\add M$ for some $M\in \Omega (\X\cap \mod_0(R))$. 
%\end{cor}     

%\begin{proof} Since $\X\cap \mod_0(R)$ is also resolving, so $\Omega (\X \cap \mod_0(R))$ is closed under finite direct sums and direct summands by \cite[Lemma 3.1]{Dey-Takahashi}, hence the claim  follows from Lemma \ref{add}.  
%\end{proof}  

\begin{cor}\label{syzygy finite rep}
Let $(R,\m,k)$ be a commutative Noetherian local ring with depth $t$ and Krull dimension $d$. The following are equivalent.
\begin{enumerate}
%\item $\Omega^n\fl(R)\subseteq |\bigoplus_{i=0}^d\Omega^ik|_r$ for some  $n\geq 0$ and $r\geq 1$.

 \item For each $n\geq d$, there exists $r\geq 1$ such that $\Omega^n\fl(R)\subseteq |\bigoplus_{i=t}^d\Omega^ik|_r$.

        \item $\Omega^n\fl(R)\subseteq |\bigoplus_{i=t}^d\Omega^ik|_r$ for some $n\geq d$ and $r\geq 1$.

         \item $\Omega^n\fl(R)\subseteq |\bigoplus_{i=0}^d\Omega^ik|_r$ for some $n\geq 0$ and $r\geq 1$.
         
    \item There exists $G\in \mod_0(R)$ such that $\Omega^n\fl(R)\subseteq [G]_r$ for some $n\geq 0$ and $r\geq 1$.

    \item $\ca(R)$ contains some power of the maximal ideal.

    \item $\widehat{R}$ has an isolated singularity.
\end{enumerate}
In particular, if $\add\Omega^n\fl(R)$ has finite representation type for some $n\geq 0$, then $\widehat{R}$ (and hence $R$) has an isolated singularity.
\end{cor}
\begin{proof}
    The equivalence $(5)\iff (6)$ follows from Theorem \ref{equivalent conditions}. Combining with Theorem \ref{equivalent conditions} again, the proof of the equivalences $(1)-(5)$ follow from the same argument in the proof of Corollary \ref{extends-BHST}.
    
    If $\add\Omega^n\fl(R)$ has finite representation type for some $n\geq 0$, then Lemma \ref{add} implies that $\add\Omega^n\fl(R)=\add M$ for some $M\in \add\Omega^n\fl(R)$. Hence, $\Omega^n\fl(R)\subseteq \add M=|M|_1\subseteq [M]_1$. The above equivalences now show that $\widehat{R}$ has an isolated singularity.
\end{proof}
%\begin{rem}
 %   Since $[G]_i$ is closed under direct summands, the condition (2) above is equivalent to that $\add\Omega^n\fl(R)\subseteq [G]_i$. 
%\end{rem}

\begin{chunk}\label{CM_0(R)=add}
Let $R$ be a Cohen--Macaulay local ring of dimension $d$.  It follows from \cite[Claim 1 in proof of Theorem 2.4]{Takahashi:2010} that 
 $$\CM_0(R)=\add\Omega^d\fl(R);$$
 see also \cite[Proposition 2.3]{MatT}.
\end{chunk}
 As mentioned before Corollary \ref{extends-BHST}, Theorem \ref{equivalent conditions} is generalizes a result of \cite[Theorem 3.2]{BHST}, it also extends a result of Dao and Takahashi \cite[Theorem 1.1 (1)]{DT2015}. Specifically,  it was proved in loc. cit. that the equivalence $(1)\iff (2)$ below holds, and all the conditions listed below are equivalent if $R$ is an equicharacteristic complete Cohen--Macaulay local ring with a perfect residue field.
 
\begin{cor}\label{extends-DT}
    Let $R$ be a Cohen--Macaulay local ring. The following are equivalent.
    \begin{enumerate}
    \item $\CM_0(R)=[G]_r$ for some $G\in\CM_0(R)$ and $r\geq 1$.
    
    \item $\ann\Ext(\CM_0(R),\CM_0(R))$ contains some power of the maximal ideal.

    \item $\widehat R$ has an isolated singularity. 
\end{enumerate}
\end{cor}

\begin{proof}
    Let $\X=\Omega^d\fl(R)$, where $d=\dim(R)$.  By \ref{CM_0(R)=add}, $\add \X=\CM_0(R)$.  Combining this with the implication  $(3)\Longrightarrow (4)$  in Theorem \ref{equivalent conditions}, we obtain $(1)\Longrightarrow (2)$. The implication $(2)\Longrightarrow (3)$ was proved in the proof of \cite[Proposition 4.9 (2)]{Dey-Takahashi2022}; this also follows from the implication $(5)\Longrightarrow (7)$ in Theorem \ref{equivalent conditions}.
    
$(3)\Longrightarrow (1)$. By the implication $(7)\Longrightarrow (8)$ in Theorem \ref{equivalent conditions}, we have $\X\subseteq |\Omega^d k|_r$ for some $r\geq 1$. Since $\Omega^dk\in\CM_0(R)$, it follows that $|\Omega^dk|_r\subseteq \CM_0(R)$. As $|\Omega^dk|_r$ is closed under direct summands, we obtain $$\CM_0(R)=\add\X\subseteq |\Omega^dk|_r,$$ and hence $\CM_0(R)=|\Omega^d k|_r$. On the other hand, $[\Omega^dk]_r\subseteq \CM_0(R)$ as $\Omega^dk\in\CM_0(R)$. Combining this with $\CM_0(R)=|\Omega^dk|_r\subseteq [\Omega^dk]_r$, we have $\CM_0(R)=[\Omega^dk]_r$. This completes the proof.
\end{proof}

\begin{rem}\label{extends-DLT}
Let $R$ be a local ring with depth $t$ and Krull dimension $d$. Assume that $\widehat R$ has an isolated singularity. By Theorem \ref{equivalent conditions} and Corollary \ref{extends-BHST},  $\Omega^n\mod(R)\subseteq |\bigoplus_{i=t}^d\Omega^ik|_r$ for some $n\geq d$ and $r\geq 1$. In particular, $\mod(R)$ has a strong generator in the sense of \cite[Section 4]{IT2016}. This is due to the first author, Lank, and Takahashi \cite[Remark 3.13]{DLT}.
Assume, in addition, that $R$ is Cohen--Macaulay, it is also proved in loc. cit. that $\CM(R)=|G|_r$ for some $G\in\mod(R)$ and $r\geq 1$. Indeed, one can take $G=\Omega^dk$; see the proof of Corollary \ref{extends-DT}. 
\end{rem}

\begin{chunk}\label{syzygy-res-closed-summands}
    Let $\X$ be a resolving subcategory of $\mod(R)$; see \ref{def of res}. By \cite[Lemma 3.1]{Dey-Takahashi}, $\Omega\X$ is closed under direct summands, and hence $\Omega\X=\add\Omega\X$. For example, $\Omega\mod(R)$ is closed under direct summands. If $R$ is Cohen--Macaulay, since $\CM(R)$ and $\CM_0(R)$ are resolving, both $\Omega\CM(R)$ and $\Omega\CM_0(R)$ are closed under direct summands. 
\end{chunk}

Next, we present several examples of local rings with finite syzygy representation type.
\begin{ex}
Let $k$ be a field.

(1) Let $R=k\llbracket  x,y\rrbracket/(x^m,xy,y^n)$, where $m,n\ge 2$.
   Then $R$ is an Artinian local ring.
   In particular, $\fl(R)=\mod(R)$ and $R$ has an isolated singularity.
   Since $R$ is not a principal ideal ring, $\mod(R)$ does not have finite representation type; see \cite[Theorem 3.3]{LW} for instance.
   On the other hand, $(x,y)=(x)\oplus(y)$ in $R$.
   By \cite[Proposition 2.5]{Takahashi-uni-dom}, for any $M\in\mod(R)$, there is an isomorphism $$\Omega_R^2M\cong\Omega_A^2(M/yM)\oplus\Omega_B^2(M/xM)\oplus (y)^{\oplus a}\oplus (x)^{\oplus b}$$ for some integers $a,b\ge 0$, where $A=R/(y)$ and $B=R/(x)$.
   Hence, $\add\Omega^2\mod(R)$ has finite representation type as $A$ and $B$ are Artinian principal ideal rings; see \cite[Theorem 3.3]{LW} again.
   When $m=n=2$, $\Omega\mod(R)$ has finite representation type because $\Omega\mod(R)$ belongs to $\mod(k)$.
 %  Whether it has a finite $\add\Omega(\mod(R))$-representation type in other cases remains an open problem.

(2) Consider the local ring $R=k\llbracket x,y\rrbracket/(x^2,xy)$.
   This is a typical non-Cohen--Macaulay local ring of depth 0 and dimension 1.
   Take an indecomposable torsionless $R$-module $M$ which is not free.
   Then, since $M$ is a submodule of a module of the form $\m R^a$ for $a\ge1$, it is killed by the element $x\in R$.
   The quotient ring $R/(x)\cong k\llbracket y\rrbracket$ is a principal ideal domain. Hence, by the structure theorem for finitely generated modules over a principal ideal domain,  $M$ is isomorphic to either  $R/(x)$ or  $R/(x,y^n)$ for some $n\geq 1$.
   But the $R$-module $R/(x,y^n)$ is not torsionless when $n\geq 2$.
   Indeed, there is an isomorphism $\Hom_R(R/(x,y^n),R)\cong \ann(x,y^n);~f\mapsto f(\overline{1})$. This implies that
   $$\Hom_R(R/(x,y^n),R)\cong (x);~f\mapsto f(\overline{1})$$ for $n\geq 2$. It follows that there is no injective map from $R/(x,y^n)$ to a free $R$-module if $n\geq 2$.  Thus, $M$ is isomorphic to either $R/(x)$ or $R/(x,y)$. Consequently, $\Omega\mod(R)$ has finite representation type.
\end{ex}

It is known that a Cohen--Macaulay local ring $R$ with finite Cohen--Macaulay type has an isolated singularity; this is due to Auslander when $R$ is complete, due to Leuschke and Wiegand when $R$ is excellent, and due to Huneke and Leuschke for the general case. Recently, Dao and Takahashi \cite[Corollary 1.2]{DT2015} improved this famous result to the following: for a Cohen--Macaulay local ring $R$, if  $\CM_0(R)$ has finite representation type, then $R$ has an isolated singularity.

The following result improves the above result of Dao and Takahashi. 

\begin{cor}\label{application}
Let $R$ be a Cohen--Macaulay local ring. Then:
\begin{enumerate}
    \item If $\OCM_0(R)$ has finite representation type, then $\widehat{R}$ has an isolated singularity, and hence $\OCM(R)$ has finite representation type.

    \item If $\add\Omega^n\CM_0(R)$ has finite representation type for some $n\geq 0$, then $\widehat{R}$ has an isolated singularity, and hence $\add\Omega^n\CM(R)$ has finite representation type.
\end{enumerate}
%consequently, $\OCM(R)$ is of finite representation type, hence the Grothendieck group $G(\mod(R))$ is finitely generated.
\end{cor}
\begin{proof}
Let $d$ denote the Krull dimension of $R$. By \ref{CM_0(R)=add}, $\CM_0(R)=\add\Omega^d\fl(R)$. This yields that $\add\Omega^n\CM_0(R)=\add\Omega^{n+d}\fl(R)$. If $\add\Omega^n\CM_0(R)$ has finite representation type, it follows from that Corollary \ref{syzygy finite rep} that $\widehat{R}$ has an isolated singularity. Consequently, $\CM(R)=\CM_0(R)$. This proves (2).

For (1),  by \ref{syzygy-res-closed-summands},
$
\OCM_0(R)=\add\OCM_0(R).
$
The desired result now follows from (2).
\end{proof}
\begin{rem}
(1) For a Cohen--Macaulay local ring $R$,  Corollary \ref{application} implies that $\widehat{R}$ has an isolated singularity if $R$ has finite Cohen--Macaulay type. This result was recently proved by the second author \cite[Remark 2.10]{Kimura-compact}.

  (2)   Let $R=k\llbracket x,y\rrbracket/(x,y)^2$, where $k$ is a field. In this example $\CM(R)=\CM_0(R)$. Note that $\OCM(R)$ consists of only one indecomposable module $k$. Thus, it satisfies the assumption of Corollary \ref{application}. It is known that $R$ has infinite Cohen--Macaulay type. 

     This example provides a Cohen--Macaulay local $R$ such that $\OCM_0(R)$ has a finite representation type, while $\CM_0(R)$ does not. In such examples, we can apply Corollary \ref{application} to get that $\widehat R$ has an isolated singularity, and in particular $R$ has an isolated singularity. However, such examples do not satisfy the assumption of Dao and Takahashi's result mentioned before Corollary \ref{application}.
\end{rem}

\begin{chunk}
   Let $\widetilde{\S}_n(R)$ (resp. $\S_n(R)$) denote the subcategory of $\mod(R)$ consisting of modules $M$ that satisfy $(\widetilde{\S}_n)$ (resp. $(\S_n)$), i.e., $\depth(M_{\p})\geq \min \{n,\depth(R_{\p})\}$ (resp. $\depth (M_{\p})\ge \min \{n,\dim (R_{\p})\}$) for all $\p \in \spec(R)$. 
   
   We also put $\widetilde{\S}_n^0(R):=\widetilde{\S}_n(R) \cap \mod_0(R)$ and $\S_n^0(R):=\S_n(R) \cap \mod_0(R)$ . 
   Note that $\widetilde{\S}_n(R)$ is always a resolving subcategory of $\mod(R)$ containing $\Omega^n\mod(R)$, and $\widetilde{\S}_n^0(R)$ is a resolving subcategory containing $\Omega^n\mod_0(R)$, and if $R$ satisfies $(\S_n)$, then $\widetilde{\S}_n(R)=\S_n(R)$.     
\end{chunk}
\begin{ex}\label{exa}
    Let $R$ be a Cohen--Macaulay local ring with Krull dimension $d$. For each $n\geq d$,
    $$
\widetilde{\S}_n(R)=\S_n(R)=\CM(R)\text{ and }\widetilde{\S}_n^0(R)=\S_n^0(R)=\CM_0(R)
$$
\end{ex}
\begin{prop}\label{equivalent conditions-Serre}
    Let $R$ be a commutative Noetherian local ring and $m,n\geq 0$. Consider the following conditions:
\begin{enumerate}
    \item  There exists $G\in\mod_0(R)$  such that $\Omega^m\widetilde{\S}_n^0(R)\subseteq [G]_r$ for some $r\geq 1$.
%\item For each subcategory $\mathcal Y$ of $\mod(R)$, the ideal $\ann\Ext_R((\Omega^t(\widetilde{\S}_n^0(R)),\mathcal{Y})$ contains some powers of the maximal ideal.

   % \item For some subcategory $\mathcal Y$ of $\mod(R)$ containing $\Omega^{t+1}(\widetilde{\S}_n^0(R))$, the ideal $\ann\Ext_R(\Omega^t(\widetilde{\S}_n^0(R)),\mathcal{Y})$ contains some powers of the maximal ideal.
 
\item  $\widehat R$ has an isolated singularity.

  %  \item  $R$ has an isolated singularity.
\end{enumerate}
Then $(1)\Longrightarrow (2)$. The implication $(2)\Longrightarrow(1)$ holds if, in addition, $m\geq \dim(R)$.
\end{prop}
\begin{proof}
Since $\Omega^n\mod_0(R)\subseteq \widetilde{\S}_n^0(R)$, we have the first inclusion below:
 $$
 \Omega^{m+n}\mod_0(R)\subseteq \Omega^m\widetilde{\S}_n^0(R)\subseteq   \Omega^m\mod(R).
  $$
The desired result now follows from Theorem \ref{equivalent conditions}.
\end{proof}
In view of Example \ref{exa}, the following result is a generalization of Corollary \ref{application} (1).
\begin{thm}\label{main} Let $R$ be a commutative Noetherian local ring. 
If $\Omega\widetilde{\S}_n^0(R)$ has finite representation type for some $n\geq 0$, then $\widehat R$ has an isolated singularity.
%% If for some integer $n\ge 0$, $\Omega \widetilde{\S}_n^0(R)$ has only finitely many indecomposable objects, then $R$ has at most isolated singularity.   
\end{thm}  
\begin{proof}
  Since both $\widetilde{\S}_n(R)$ and $\mod_0(R)$ are resolving subcategories, their intersection $\widetilde{\S}_n^0(R)$ is also resolving.  By \ref{syzygy-res-closed-summands}, $\Omega\widetilde{\S}_n^0(R)$ is closed under direct summands. As this category has finite representation type, Lemma \ref{add} implies that  $\Omega\widetilde{\S}_n^0(R)=\add(G)$ for some $G\in \Omega\widetilde{\S}_n^0(R)$. Noting that $\add(G)\subseteq [G]_1$ and $G\in \mod_0(R)$, we can apply Proposition \ref{equivalent conditions-Serre} to conclude that $\widehat R$ has an isolated singularity.
\end{proof}

\section{Schreyer-type questions}\label{Section-ascent-descent}

%For a local ring $R$, the objects of $\fl(R)$ and $\fl(\widehat{R})$ can be identified.
%From this perspective, $\add\Omega^n \fl(R)$ and $\add\Omega^n \fl(\widehat{R})$ are expected to behave similarly for each $n$.

The main result in this section is Theorem \ref{T2} from the introduction. We also obtain a syzygy version of a result of Kobayashi, Lyle, and Takahashi \cite[Theorem 4.4]{KLT}; see Proposition \ref{generalization-KLT}. In the following, we will study the relationships between certain subcategories of $\mod(R)$ and subcategories of $\mod(\widehat{R})$ by comparing $\add\Omega^n \fl(R)$ and $\add\Omega^n \fl(\widehat{R})$.
We prepare a basic lemma for this purpose.

\begin{lem}\label{count. iff}
Let $R$ be a commutative Noetherian local ring and $\X$ be a subcategory of $\mod(R)$.
The following are equivalent.
\begin{enumerate}
 \item $\add\X$ has countable representation type.
 \item There are only countably many isomorphism classes of modules belonging to $\add\X$.
 \item There are only countably many isomorphism classes of modules belonging to $\X$.
\end{enumerate}
\end{lem}

\begin{proof}
The implication $(2)\Longrightarrow(3)$ is straightforward.

$(1)\Longrightarrow(2)$. Suppose that $R$ has countable $\add\X$-representation type.
Let $X_1, X_2, \ldots $ be all representatives of isomorphism classes of indecomposable modules belonging to $\add\X$.
We denote by $\mathbb{N}$ the set of non-negative integers.
For any $m\in\mathbb{N}$, there exists a surjection $\mathbb{N}\to\mathbb{N}^m;~n\mapsto (a_1^{(m,n)},\ldots,a_m^{(m,n)})$.
Then the map $$\mathbb{N}^2\to\add\X;~ (m,n) \mapsto X_1^{a_1^{(m,n)}}\oplus\cdots\oplus X_m^{a_m^{(m,n)}}$$ is surjective up to isomorphism.

$(3)\Longrightarrow(1)$. Assume that there are only countably many isomorphism classes of modules belonging to $\X$.
Put $\Y=\{X_1^{a_1}\oplus\cdots\oplus X_m^{a_m}\mid m>0, X_i\in\X, a_i\in\mathbb{N}$ for all $1\le i\le m\}$.
Using the same argument as in the previous paragraph, it follows that there are only countably many isomorphism classes of modules belonging to $\Y$.
For each $Y\in\Y$, there are, up to isomorphism, only finitely many indecomposable modules which are direct summands of $Y$ by \cite[Theorem 1.1]{Wiegand}.
Any indecomposable module belonging to $\add\X$ is isomorphic to a direct summand of some module belonging to $\Y$, which means that $R$ has countable $\add\X$-representation type.
\end{proof}

%Thanks to Corollary \ref{syzygy finite rep}, the existence of $G\in\mod_0 R$ such that  $\add\Omega^n\fl(R) \subset [G]_r$ for some $r>0$ is equivalent to the existence of $H\in\mod_0 \widehat{R}$ such that  $\add\Omega^n\fl(\widehat{R}) \subset [H]_s$ for some $s>0$.
The following result concerns the ascent and descent of finite and countable representation types for the categories of syzygies of finite length modules.

\begin{prop}\label{flR type and comp}
Let $R$ be a commutative Noetherian local ring and $n\geq 0$ be an integer.
Then $\add\Omega^n\fl(R)$ has finite (resp. countable) representation type if and only if $\add\Omega^n\fl(\widehat{R})$ has finite (resp. countable) representation type.
\end{prop}

\begin{proof}
First, we deal with the case of finite representation types.
The ``if'' part is an immediate consequence of \cite[Theorem 1.4]{Wiegand}.
Any $\widehat{R}$-module belonging to $\fl(\widehat{R})$ is in $\fl(R)$ as an $R$-module.
For all $M\in\add\Omega^n\fl(\widehat{R})$, there exists $N\in\Omega^n\fl(R)$ such that $M$ is a direct summand of $N\otimes_R\widehat{R}$.
The ``only if'' part now follows from \cite[Lemma 2.1]{Wiegand}.

Next, we handle the case of countable representation types.
Since the objects of $\fl(R)$ and $\fl(\widehat{R})$ correspond one-to-one, there is a natural one-to-one correspondence between the sets of isomorphism classes of modules in $\Omega^n\fl(R)$ and $\Omega^n\fl(\widehat{R})$ by \cite[Corollary 1.15]{LW}.
The assertion now follows from Lemma \ref{count. iff}.
\end{proof}

%For a Cohen--Macaulay local ring $R$ of dimension $d$ and $n\ge 0$, the equality $\Omega^n \CM_0 (R)=\add\Omega^{n+d}\fl(R)$ holds; see the proof of Theorem \ref{cm category version}.
%In particular, the following equalities hold when $n=0,1$.

\begin{prop}\label{ttrf}
Let $(R,\m)$ be a commutative Noetherian local ring with Krull dimension $d$. Then:
\begin{enumerate}
    \item 
$\TF_n^0(R)=\widetilde{\S}^0_n(R)=\left\{M\in\mod_0(R)\mid\depth (M)\geq n \right\}$ for each $0\leq n\leq \depth(R)$.
\item $\TF_{n+1}^0(R)=\Omega\TF_n^0(R)=\Omega\widetilde{\S}^0_n(R)$ for each $0\leq n\leq \depth(R)$.
    \item 
   If, in addition,  $R$ is Cohen--Macaulay, then $$\TF_d^0(R)=\widetilde{\S}_d^0(R)=\CM_0(R) \text{ and } \TF_{d+1}^0(R)=\Omega\widetilde{\S}_d^0(R)=\Omega\CM_0(R).$$
\end{enumerate}
\end{prop}

\begin{proof}
The statement (3) follows immediately from (1) and (2). Next,  assume $0\leq n\leq \depth(R)$.

(1) The equality $\widetilde{\S}^0_n(R)=\left\{M\in\mod_0R\mid\depth (M)\geq n\right\}$ is by the definition of $\widetilde{\S}_n^0(R)$. Since $n\leq \depth(R)$, $\m\notin\{\p\in \Spec(R)\mid \depth(R_\p)\leq n-1\}$. In particular, for each $M\in\mod_0(R)$, $M_\p$ is free for each $\p\in \{\p\in \Spec(R)\mid \depth(R_\p)\leq n-1\}$. By \cite[Proposition 2.4 (b)]{DS},  $\TF_n^0(R)=\widetilde{\S}^0_n(R)$.

(2) The equality  $\Omega\TF_n^0(R)=\Omega\widetilde{\S}^0_n(R)$ follows from (1). It remains to prove $\TF_{n+1}^0(R)=\Omega\TF_n^0(R)$.

For each $M\in \TF_n^0(R)$, take a minimal free resolution of $M^\ast$ as follows:
$$
 F_{n}\xrightarrow{d_{n}}F_{n-1}\rightarrow \cdots \rightarrow F_1\xrightarrow {d_1} F_0\rightarrow M^\ast\rightarrow 0.
$$
By the hypothesis, $\Ext^i_R(\Tr_RM,R)=0$ for all $1\le i<n+1$. This implies that $M\cong M^{\ast\ast}$ and $\Ext^i_R(M^\ast,R)=0$ for $1\leq i\leq n-1$ (see \ref{def-torsionfree}),  and hence the above free resolution of $M^\ast$ induces an exact sequence:
$$
0\rightarrow M \rightarrow F_0^\ast\xrightarrow{d_{1}^\ast}F_1^\ast\rightarrow \cdots\rightarrow F_{n-1}^\ast\xrightarrow{d_{n}^\ast}F_n^\ast\rightarrow C\rightarrow 0
$$ where $C=\cok(d_{n}^\ast)$. It follows that $M=\Omega^{n+1}C$. We claim that $C\in \mod_0(R)$. Indeed, since $M\in\mod_0(R)$, one has $M^\ast\in\mod_0(R)$ as well. It follows that the free resolution of $M^\ast$ above splits after localization at any non-maximal prime ideal. In particular, for each $\p\in \Spec(R)\setminus\{\m\}$, $\Ker(d_n)_\p$ is projective and $C_\p\cong \Ker(d_n)_\p^\ast$. It follows that $C\in\mod_0(R)$, and hence $\Ima(d_1^\ast)=\Omega^nC$ is in $\mod_0(R)$. Moreover, $\depth(\Omega^n N)\geq n$ by the depth formula; see \cite[1.3.7]{BH}. We conclude that $\Ima(d_1^\ast)=\Omega^n(C)\in \TF_n^0(R)$ by (1). This implies that $M=\Omega\Ima(d_1^\ast)\in\Omega\TF_t^0(R)$.

Conversely, let $M\in \Omega\TF_n^0(R)$. By \cite[Theorem 2.17]{AB}, any $n$-torsionfree module is a $n$-th syzygy module. It follows that
 $M$ is an $(n+1)$-th syzygy. Combining with the assumption that $M$ is locally free on the punctured spectrum, we deduce from \cite[Theorem 43]{Masek} that $M$ is $(n+1)$-torsionfree. Thus, $M\in\TF_{n+1}^0(R)$. This completes the proof.
\end{proof}
\begin{rem}
   Proposition \ref{ttrf} (1) strengthens the first equality in \cite[Corollary 4.3]{Dey-Takahashi}, which was proved by the first author and Takahashi.
\end{rem}
\begin{cor}\label{tf-finite representation type}
    Let $R$ be a commutative Noetherian local ring. If $\TF_n^0(R)$ has finite representation type for some $n\leq \depth(R)+1$, then $\widehat{R}$ has an isolated singularity.
\end{cor}
\begin{proof}
    If $n\leq \depth(R)$ and $\TF_n^0(R)$ has finite representation type, then Proposition \ref{ttrf} implies that $\widetilde{\S}_n^0(R)$ also has finite representation type. Since $\Omega\widetilde{\S}_n^0(R)\subseteq \widetilde{\S}_n^0(R)$, it follows that $\Omega\widetilde{\S}_n^0(R)$ has finite representation type as well. By Theorem \ref{main}, we get that $\widehat{R}$ has an isolated singularity. 

     If $n= \depth(R)+1$ and $\TF_n^0(R)$ has finite representation type, the conclusion follows directly from Proposition \ref{ttrf} and Theorem \ref{main}.
\end{proof}
\begin{rem}
 (1)  Let $t$ denote the depth of a local ring $R$. For each $n\geq t$, 
   $$\widetilde{\S}_n^0(R)=\{M\in\mod_0(R)\mid \depth(M)\geq t\}=\widetilde{\S}_t^0(R).$$
Combining with this, Proposition \ref{ttrf} shows that Theorem \ref{main} and Corollary \ref{tf-finite representation type} are equivalent statements.

(2) If $0\leq m\leq n$, then $\TF_n^0(R)\subseteq \TF_m^0(R)$. It follows that Corollary \ref{tf-finite representation type} is also equivalent to the following:  Let $R$ be a commutative Noetherian local ring with depth $t$. If $\TF_{t+1}^0(R)$ has finite representation type, then $\widehat{R}$ has an isolated singularity.
\end{rem}
\begin{chunk}
 For a Cohen--Macaulay local ring $R$,  Schreyer conjectured in \cite[Conjecture 7.3]{Sch} that
\begin{center}
    \emph{$R$ has finite CM type if and only if the $\m$-adic completion $\widehat{R}$ has finite CM type.}
\end{center} 
Leuschke and Wiegand \cite[Main Theorem]{LW2000} established that the conjecture holds if, in addition, $R$ is excellent. They also proved that the conjecture fails without the Cohen--Macaulay assumption; see \cite[Examples 2.1 and 2.2]{LW2000}.
\end{chunk}
The case $n = 0$ in statement (2) of the following theorem provides an affirmative answer to Schreyer's conjecture.

\begin{thm}\label{cm category version}
Let $R$ be a Cohen--Macaulay local ring with Krull dimension $d$. For each $n\geq 0$, we have:

\begin{enumerate}
 \item $\add\Omega^n\CM_0 (R)$ has finite (resp. countable) representation type if and only if $\add\Omega^n\CM_0 (\widehat{R})$ has finite (resp. countable) representation type.
 
 \item $\add\Omega^n\CM(R)$ has finite representation type if and only if $\add\Omega^n\CM (\widehat{R})$ has finite representation type.

 %\item Suppose that $n\in\{d, d+1\}$, where $d=\dim R$.
 %Then $\TF_n^0(R)$ has finite representation type if and only if $\TF_n^0(\widehat{R})$ has finite representation type.
 
 \item Suppose that $n\in\{d, d+1\}$.
 Then $\TF_n(R)$ has finite representation type if and only if $\TF_n(\widehat{R})$ has finite representation type.

% \item $\add\Omega^n\CM_0 (R) \subset [G]_r$ for some $r>0$ and $G\in\mod_0 R$ if and only if $\add\Omega^n\CM_0 \widehat{R} \subset [H]_s$ for some $s>0$ and $H\in\mod_0 \widehat{R}$.
\end{enumerate}
%In particular, if one of the conditions in {\rm (1)}, {\rm (2)}, or {\rm (4)} holds, $\widehat{R}$ has an isolated singularity.
\end{thm}

\begin{proof}
(1) For any Cohen--Macaulay local ring $A$ of dimension $d$, $\CM_0 (A)=\add\Omega^d\fl(A)$ by \ref{CM_0(R)=add}, and hence we have $\add\Omega^n\CM_0 (A)=\add\Omega^{n+d}\fl(A)$.
Therefore, it follows from Corollary \ref{syzygy finite rep} and Proposition \ref{flR type and comp} that the statement (1) holds.

(2) Let $A\in\{R,\widehat{R}\}$.
If $\add\Omega^n\CM (A)$ has finite representation type, then Corollary \ref{application} implies that $\widehat{R}$ has an isolated singularity. Consequently, $\CM(R)=\CM_0(R)$ and $\CM(\widehat{R})=\CM_0(\widehat{R})$.  Thus, the equivalence in (2) follows from (1).

(3) Let $A\in\{R,\widehat{R}\}$. By Proposition \ref{ttrf}, $\TF_d^0(A)=\CM_0(A)$ and $\TF_{d+1}^0(A)=\OCM_0(A)$. Moreover, $\CM_0(A)$ and $\OCM_0(A)$ are closed under direct summands; see \ref{syzygy-res-closed-summands} for the second one. Thus, (1) implies that $\TF_n^0(R)$ has finite representation type if and only if $\TF_n^0(\widehat{R})$ has finite representation type.  Assume $\TF_n(A)$ has finite representation type, then Corollary \ref{tf-finite representation type} implies that $\widehat{R}$ has an isolated singularity, and hence $\CM(A)=\CM_0(A)$. Combining with $\TF_n(A)\subseteq \CM(A)$, we conclude that $\TF_n(R)=\TF_n^0(R)$ and $\TF_n(\widehat{R})=\TF_n^0(\widehat{R})$. Thus, the equivalence in (3) holds.
%If $A$ has finite $\add\Omega^n\CM (A)$-representation type, $A$ also has finite $\add\Omega^n\CM_0 (A)$-representation type.
%Thus $B$ has finite $\add\Omega^n\CM_0(B)$-representation type and an isolated singularity by Corollary \ref{syzygy finite rep}.
%This means that $B$ has finite $\add\Omega^n\CM (B)$-representation type and that (1) holds.
%Assume that $A$ has finite $\TF_n(A)$-representation type.
%Then $A$ has finite $\TF_n^0 A$-representation type.
%By Proposition \ref{ttrf}, similarly, we see that $B$ has finite $\TF_n^0 B$-representation type and an isolated singularity, which means that $B$ has finite $\TF_n(B)$-representation type and that (2) holds.
\end{proof}
\begin{rem}\label{all are equivalent}
(1) Let $R$ be a Cohen--Macaulay local ring. Combining with Corollary \ref{application}, parts (1) and (2) of Theorem \ref{cm category version} imply that, for each $n\in\{0,1\}$, the following are equivalent. 
\begin{itemize}
    \item $\Omega^n\CM_0(R)$ has finite representation type.

    \item $\Omega^n\CM_0(\widehat{R})$ has finite representation type.

 \item $\Omega^n\CM(R)$ has finite representation type. 
 
    \item  $\Omega^n\CM(\widehat{R})$ has finite representation type.
\end{itemize}

(2) For a Cohen--Macaulay local ring $R$, the fact that finite Cohen--Macaulay type descends from $\widehat{R}$ to $R$ is due to Wiegand \cite{Wiegand}.
Conversely, along the same line as the proof of Theorem \ref{cm category version} (2), one can show that finite Cohen--Macaulay type ascends from $R$ to $\widehat{R}$ by combining Takahashi’s result \cite[Corollary 3.6]{Takahashi:2010} with a recent observation of the second author \cite[Remark 2.10]{Kimura-compact}.
\end{rem}
\Cref{C1} (2) follows from the following result. For the details about Grothendieck groups, we refer the reader to \cite[Chapter 13]{Yoshino}.
\begin{prop}\label{dim2} 
%\old{Let $R$ be a Henselian excellent Cohen--Macaulay local ring of dimension $2$, where $k$ is an algebraically closed field. If $\OCM_0(R)$ has finite representation type, then $R$  has minimal multiplicity.}
%\old{Let $R$ be a Cohen--Macaulay local ring and suppose that $\add\syz^n  \fl(R)$ has finite representation type for some $n\geq 0$. Then the Grothendieck group of $\mod(\widehat R)$ is finitely generated. In particular, if, in addition, $\dim(R)=2$ and the residue field of $R$ is algebraically closed, then $\Spec(\widehat R)$ has rational singularities, and consequently $R$ has minimal multiplicity.}

Let $R$ be a Cohen--Macaulay local ring. Consider the following conditions:
\begin{enumerate}
    \item $\syz \CM(R)$ has finite representation type.

    \item $\add\syz^n  \fl(R)$ has finite representation type for some $n\geq 0$. 

    \item $\widehat R$ has isolated singularity and the Grothendieck group of $\mod(\widehat R)$ is a finitely generated Abelian group. 

    \item $\Spec(\widehat R)$ has isolated rational singularities (cf. \cite[Definition 6.32]{LW}). 
\end{enumerate}

Then $(1)\Longrightarrow (2)\Longrightarrow (3)$. If, in addition, $\dim(R)=2$ and the residue field of $R$ is algebraically closed, then all the above conditions are equivalent, and they imply that $R$ has minimal multiplicity. 

\end{prop} 

\begin{proof}
Let $d$ denote the Krull dimension of $R$.

$(1)\Longrightarrow (2)$. By \cite[Lemma 3.1(2)]{Dey-Takahashi2022}, we have $\add \syz^{d+1}\fl(R)\subseteq \syz \CM(R)$, hence the claim.
 
$(2)\Longrightarrow (3)$. By \Cref{flR type and comp}, $\add\syz^n \fl(\widehat R)$ has finite representation type. Hence, we may assume $R$ is complete. By \Cref{syzygy finite rep}, $R$ has an isolated singularity. By \cite[4.6]{DLMO}, $n\geq  d$. It follows from \Cref{CM_0(R)=add} that $\add \syz^{n}\CM(R)$ has finite representation type. Let $H_1,...,H_l$ be these finitely many indecomposable objects. Then, for each $M\in \mod(R)$, there exists non-negative integers $a_1,...,a_l$ such that $\syz^{n} M\cong \bigoplus_{i=1}^l H_i^{\oplus a_i}$. Since $[M]\in \mathbb Z[R]+\mathbb Z[\syz^n M]$ in the Grothendieck group, we conclude that the group is finitely generated as an Abelian group by $[R],[H_1],...,[H_l]$.

Next, assume, in addition,  $\dim(R)=2$ and the residue field of $R$ is algebraically closed. 

$(3)\Longrightarrow (4)$.  Note that $\widehat R$ is excellent, Henselian, and has the same residue field as $R$. Since $\dim \widehat R= 2$,  $\widehat R$ satisfies $(R_1)$, i.e., $\widehat R_\p$ is regular for each prime ideal $\p$ of $\widehat R$ with height at most $1$. Since $\widehat R$ is also Cohen--Macaulay, by Serre's criteria, $\widehat R$ is normal. Combining with the fact that $\widehat R$ is local, we conclude that $\widehat R$ is an integral domain. So, $\Spec(\widehat R)$ has rational singularities by \cite[Corollary 3.3]{groth}. 

$(4)\Longrightarrow (1)$. Note that $\widehat R$ is excellent, Henselian, and has the same residue field as $R$. Since $\dim \widehat R= 2$,  $\widehat R$ satisfies $(R_1)$, i.e., $\widehat R_\p$ is regular for each prime ideal $\p$ of $\widehat R$ with height at most $1$. Since $\widehat R$ is also Cohen--Macaulay, by Serre's criteria, $\widehat R$ is normal. Combining with the fact that $\widehat R$ is local, we conclude that $\widehat R$ is an integral domain. By \cite[Corollary 3.3]{groth}, $\syz \CM(\widehat R)$ has finite representation type. The desired result of (1) now follows from Theorem \ref{cm category version}. 

Finally, by \cite[Corollary 6.36]{LW}, if one of the above equivalent conditions holds, then $\widehat R$ has minimal multiplicity, and hence so does $R$.
\end{proof}

  At the end of this section, we consider a syzygy version of $\CM_+ (R)\colonequals \CM(R)\setminus \CM_0(R)$ studied by Kobayashi, Lyle, and Takahashi \cite{KLT}. Let $\X$ be a full subcategory of $\mod(R)$,  set $\X_0\colonequals\X\cap \mod_0(R)$ and $\X_{+}\colonequals\X\setminus\X_0$.

\begin{lem}\label{finite set}
Let $(R,\m)$ be a commutative Noetherian local ring and $\X$ a full subcategory of $\mod(R)$.
Suppose that for each $\p\in\Sing (R)\setminus\{\m\}$, there exists $n\geq 0$ such that $\Omega^n(R/\p)\in\X$.
If $(\add\X)_{+}$ has finite (resp. countable) representation type, then $\Sing (R)$ is a finite (resp. countable) set.
\end{lem}

\begin{proof}
Fix $\p\in\Sing (R)\setminus\{\m\}$.
Since $\Omega_R^n(R/\p)$ is in $\X$ for some $n\ge 0$, there are indecomposable modules $X_1,\ldots,X_m$ belonging to $(\add\X)_{+}$, positive integers $a_1,\ldots,a_m$, and $Y\in\mod_0 R$ such that $$\Omega^n(R/\p)\cong X_1^{\oplus a_1}\oplus\cdots\oplus X_m^{\oplus a_m}\oplus Y.$$
Noting that $\p$ belongs to $\Sing (R)$, we get the equality below:
$$
\p=\ann\Tor^R_1(\Omega^n(R/\p),\Omega^n(R/\p))=\bigcap_{1\le i,j\le m}\ann\Tor^R_1(X_i,X_j)\cap \ann\Tor^R_1(Y,\Omega^n(R/\p)),
$$
where the second equality is by the decomposition of $\Omega^n(R/\p)$.
Since $\ann\Tor^R_1(Y,\Omega^n(R/\p))$ contains some power of the maximal ideal of $R$, we conclude that $\p$ is equal to $\ann\Tor^R_1(X_i,X_j)$ for some $1\le i,j\le m$.
It follows that only finite (resp. countable) ideals are represented in this way.
\end{proof}

The following result is a direct consequence of Lemma \ref{finite set}. It is a syzygy version of \cite[Theorem 4.2]{KLT} and its proof is similar as well.

\begin{cor}
Let $R$ be a commutative Noetherian local ring and $n\ge 0$. Then:
\begin{enumerate}
 \item If $(\add\Omega^n\mod(R))_{+}$ has finite (resp. countable) representation type, then $\Sing (R)$ is a finite (resp. countable) set.
 \item If $R$ is Cohen--Macaulay and $(\add\Omega^n\CM(R))_{+}$ has finite (resp. countable) representation type, then $\Sing (R)$ is a finite (resp. countable) set.
\end{enumerate}
\end{cor}

The following result is a syzygy version of \cite[Theorem 4.4]{KLT}.
However, since $\Omega^n\CM(R)$ is not necessarily closed under direct summands for $n\ge 2$, the same proof does not work.
We were able to address that issue by utilizing \cite[Theorem 1.1]{Wiegand}.

\begin{prop}\label{generalization-KLT}
Let $(R,\m)$ be a Cohen--Macaulay local ring with a canonical module $\omega$.
Suppose that $(\add\Omega^n\CM(R))_{+}$ has finite (resp. countable) representation type.
Then $\add\Omega^n\CM(R_\p)$ has finite (resp. countable) representation type for each $\p\in\Spec(R)\setminus\{\m\}$.
\end{prop}

\begin{proof}
Let $\p$ be a non-maximal prime ideal of $R$ and let $M\in\add\Omega^n\CM(R_\p)$.
There exists an $R$-module $L$ such that $L_\p\in\CM(R_\p)$ and $M$ is a direct summand of $\Omega_{R_\p}^n L_\p$.
Since $R$ admits a canonical module, we can take a maximal Cohen--Macaulay approximation $0\to Y\to X\to L\to 0$ of $L$, where $X\in\CM(R)$ and $Y\in\mod(R)$ has finite injective dimension.
The exact sequence $0\to Y_\p\to X_\p\to L_\p\to 0$ splits as $L_\p\in\CM(R_\p)$ and $Y_\p$ has finite injective dimension over $R_\p$; see \cite[3.1.24]{BH}.
This implies that $M$ is a direct summand of $\Omega_{R_\p}^n X_\p$.
The $R$-module $\Omega_R^n X$ is in $\Omega^n\CM(R)$.
There are indecomposable modules $K_1,\ldots,K_m$ belonging to $(\add\Omega^n\CM(R))_{+}$, positive integers $a_1,\ldots,a_m$, and $N\in\mod_0 R$ such that $$\Omega_R^n X\cong K_1^{\oplus a_1}\oplus\cdots\oplus K_m^{\oplus a_m}\oplus N.$$
We obtain $M\in\add \{(K_1)_\p\oplus\cdots\oplus (K_m)_\p\oplus R_\p\}$.
It follows from \cite[Theorem 1.1]{Wiegand} that there are, up to isomorphism, only finitely many indecomposable $R$-modules in the subcategory $\add \{(K_1)_\p\oplus\cdots\oplus (K_m)_\p\oplus R_\p\}$ of $\mod(R_\p)$; see also Lemma \ref{add}.
On the other hand, there are only finitely (resp. countably) many subcategories of $\mod(R_\p)$ defined in such a way since $(\add\Omega^n\CM(R))_{+}$ has finite (resp. countable) representation type.
\end{proof}

\section{On local rings of finite Cohen--Macaulay type}\label{Section-rings-F-CM type}

In this section, we study the questions of Chen \cite{Chen-arXiv} in the context of Cohen–Macaulay type; see \ref{context-CM}. The main result in this section is Theorem \ref{Main-finite CM type}.

\begin{chunk}\label{Chen-questions}
   Let $A$ be an Artin algebra  over a commutative Artinian ring. In \cite[Appendix C]{Chen-arXiv}, Chen raised the following questions:
\begin{enumerate}
    \item  Is it true that $\Gproj(A)$ has finite representation type if and only if every left Gorenstein projective $A$-module is a direct sum of finitely generated ones?

    \item  If $\Gproj(A)$ has finite representation type, is $A$ necessarily virtually Gorenstein (c.f. \ref{defofvirtuallyG})?

    \item  If $\Gproj(A)=\proj(A)$, does it follow that  $\GProj(A)=\Proj(A)$?
\end{enumerate}
Here, $\Proj(A)$ (resp. $\proj(A))$ is the category of left (resp.  finitely generated left) projective  $A$-modules, and $\GProj(A)$ (resp. $\Gproj(A))$ is the category of left (resp.  finitely generated left) Gorenstein projective  $A$-modules. 

As shown by Chen \cite[Main Theorem]{Chen:2008}, Question (1) has an affirmative answer when $A$ is a Gorenstein Artin algebra. This result can be viewed as a Gorenstein analogue of Auslander's celebrated theorem,
which establishes that an Artin algebra $A$ is of finite representation type (i.e., there are only finitely many indecomposable finitely generated left $A$-modules up to isomorphisms) if and
only if every left $A$-module is a direct sum of finitely generated modules; see \cite{Auslander:1974, Auslander:1982}. In \cite[Theorem 4.10]{Beligiannis:adv}, Beligiannis observed that an Artin algebra $A$ satisfies the property that every left Gorenstein projective $A$-module is a direct sum of finitely generated ones if and only if $A$ is virtually Gorenstein and $\Gproj(A)$ has finite representation type.  
    
It is worth noting that a positive answer to Question~(2) was proposed in Beligiannis’s article \cite[Example 8.4(2)]{Beligiannis}. However, as Beligiannis later pointed out to Chen, the argument is incorrect; see \cite[Appendix C]{Chen-arXiv}. Therefore, both Question~(1) and Question~(2) remain open.
Also, note that if Question (1) is true, then so is Question (3).
\end{chunk}
\begin{chunk}\label{context-CM}
    Let $R$ be a Cohen--Macaulay local ring. Motivated by the questions of Chen in \ref{Chen-questions}, 
       we raise the following questions:
\begin{enumerate}[label=(\alph*)]
    \item  If $R$ has finite Cohen--Macaulay type, is every left Gorenstein projective $A$-module necessarily a direct sum of finitely generated ones?

    \item  If $R$ has finite Cohen--Macaulay type, is $R$ necessarily virtually Gorenstein?

    \item  If $R$ has finite Cohen--Macaulay type and $\Gproj(R)=\proj(R)$, does it follow that  $$\GProj(R)=\Proj(R)?$$
\end{enumerate}
%In \cite[Theorem 4.20]{Beligiannis:adv}, Beligiannis observed 
\end{chunk}

The following theorem is the main result of this section, and its proof will be presented at the end. As applications, Questions (a), (b), and (c) in \ref{context-CM} have affirmative answers when $R$ is additionally assumed to be complete; for (b), see Corollary \ref{answer-b}.
\begin{theorem}\label{Main-finite CM type}
 Let $R$ be a Cohen--Macaulay local ring. Then:
 \begin{enumerate}

 \item If $\CM_0(R)$ has finite representation type, then either $R$ is a hypersurface or $\GProj(R)=\Proj(R)$.

 \item If $\CM_0(R)$ has finite representation type, then every Gorenstein projective $\widehat R$-module is a direct sum of finitely generated ones. 
 
     \item If $\GProj(R)\neq \Proj(R)$, then the following are equivalent.
     
    \begin{enumerate}
        \item $\CM(R)$ has finite representation type.

        \item $\CM_0(R)$ has finite representation type.

        \item $R$ is a hypersurface and every Gorenstein projective $\widehat R$-module is a direct sum of finitely generated ones. 

     %   \item $R$ is a Gorenstein and every Gorenstein projective $\widehat R$-module is a direct sum of finitely generated ones. 
    \end{enumerate}
 
 \end{enumerate}
%If any of the above equivalent conditions holds, then either $R$ is a hypersurface or $\GProj(R)=\Proj(R)$. 
\end{theorem}
\begin{rem}
(1) For a commutative Noetherian local ring $R$, if $\Gproj(R)$ has finite representation type, Christensen, Piepmeyer, Striuli, and Takahashi \cite[Theorem B]{CPST} established that either $R$ is Gorenstein or $\Gproj(R)=\proj(R)$. In particular, when $R$ is Cohen--Macaulay and has finite Cohen--Macaulay type, it follows that either $R$ is Gorenstein or $\Gproj(R)=\proj(R)$. It is known that a Gorenstein local ring with finite Cohen--Macaulay type is a hypersurface; see \cite[Theorem 9.15]{LW} or \Cref{herzog}. If $\GProj(R)=\GProj(R)$, then in particular  $\Gproj(R)=\proj(R)$. Whether the converse holds remains unknown. Therefore, the significance of Theorem \ref{Main-finite CM type} (1) lies in its conclusion about big modules.

 (2) Assume, in addition, $R$ is a complete Gorenstein local ring. In this case, the equivalence $(a)\iff (c)$ in Theorem \ref{Main-finite CM type} (3) is due to Beligiannis \cite[Theorem 4.20]{Beligiannis:adv}, and this result is used in the proof of Theorem \ref{Main-finite CM type}. More recently, Beligiannis's result just mentioned was recovered by Bahlekeh, Fotouhi,  and Salarian \cite[Theorem 6.8]{BFS} using a different method. 
\end{rem}

%As an immediate consequence of Theorem \ref{Main-finite CM type}, we have:
%\begin{cor}\label{direct sum of Gproj(R)}
%    Let $R$ be a complete Cohen--Macaulay local ring. If $\CM_0(R)$ has finite representation type, then every Gorenstein projective $R$-module is a direct sum of finitely generated ones.
%\end{cor}

%As an immediate consequence, we get the following result which is the second statement of Theorem \ref{T3} (2).
%\begin{cor}
 %   Let $R$ be a Cohen--Macaulay local ring. If $R$ is Gorenstein and every Gorenstein projective $\widehat R$-module is a direct sum of finitely generated ones, then $\widehat R$ has finite Cohen--Macaulay type. 
%\end{cor}

 \begin{chunk}\label{IK-equivalence}
 Let $R$ be a commutative Noetherian ring.
 %If $R$ is Gorenstein and every Gorenstein projective $\widehat R$-module is a direct sum of finitely generated ones, then 

 (1) Let $\underline{\GProj}(R)$ denote the stable category of $\GProj(R)$ modulo projective modules. The objects of $\underline{\GProj}(R)$ are the same as those of $\GProj(R)$. For any $M, N \in \underline{\GProj}(R)$, the morphism space
 $$
 \Hom_{\underline{\GProj}(R)}(M,N)\colonequals\Hom_R(M,N)/{\mathcal P}(M,N),
 $$
 where $\mathcal{P}(M,N)$ consists of  morphisms in $\Hom_R(M,N)$ which factor through a projective module. Similarly, let $\overline{\GInj}(R)$ denote the stable category of $\GInj(R)$ modulo injective modules; see \cite[Section 7]{Krause}.

(2)    Assume, in addition, $R$ has a dualizing complex $D$.  In \cite[Theorems 5.3 and 5.4]{Iyengar-krause}, Iyengar and Krause observed that $\underline{\GProj}(R)$ and $\overline{\GInj}(R)$ are compactly generated triangulated categories and there are triangle equivalences up to direct summands
     $$
  \D^f(R)/\thick_{\D^f(R)}(R,D)\xrightarrow \simeq (\underline{\GProj}(R)^c)^{\rm op} \text{ and }  \D^f(R)/\thick_{\D^f(R)}(R,D)\xrightarrow \simeq \overline{\GInj}(R)^c,
  $$ 
  where $\T^c$ is the full subcategory of compact objects of a triangulated category $\T$.
The statements in \cite{Iyengar-krause} are formulated in terms of the homotopy category of totally acyclic complexes of projective (resp. injective) pmodules. These can be reformulated as stated here, using the classical triangle equivalences between the stable category of Gorenstein projective (resp. Gorenstein injective) $R$-modules and the homotopy category of totally acyclic complexes of projective (resp. injective) modules; see \cite[Proposition 7.2]{Krause}.
 %As a consequence of the above equivalences, there is an equivalence
  %  $$
   % \overline{\GInj}(R)^c\xrightarrow \simeq (\underline{\GProj}(R)^c)^{\rm op}.
  %  $$
\end{chunk}
 Combining the following result with Theorem \ref{Main-finite CM type} (1), we obtain Corollary \ref{answer-b}. It gives an affirmative answer to Question (b) in \ref{context-CM}, provided that $R$ is complete.
\begin{lem}\label{freeimpliesvirtuallyG}
   Let $R$ be a commutative Noetherian ring with a dualizing complex. If $\GProj(R)=\Proj(R)$, then $R$ is virtually Gorenstein. 
\end{lem}
\begin{proof}
By the assumption $\GProj(R)=\Proj(R)$, we have $\underline{\GProj}(R)=0$. Combining this with two equivalences in \ref{IK-equivalence}, we have $\overline{\GInj}(R)^c=0$. This implies that $\overline{\GInj}(R)=0$ as $\overline{\GInj}(R)$ is compactly generated; see \ref{IK-equivalence}. That is, $\GInj(R)=\Inj(R)$, and hence $^\perp \GInj(R)=\Mod(R)$. 
Note that the assumption yields that $\GProj(R)^\perp=\Mod(R)$.
Thus, $\GProj(R)^\perp=\Mod(R)=^\perp \GInj(R)$. In particular, $R$ is virtually Gorenstein.
\end{proof}

\begin{rem}
 If $R$ is Artinian,  Lemma \ref{freeimpliesvirtuallyG} was proved by Beligiannis \cite[Section 8]{Beligiannis}, who also showed that it holds for any Artinian algebra.
\end{rem}

\begin{cor}\label{answer-b}
    Let $R$ be a complete Cohen--Macaulay ring. If $\CM_0(R)$ has finite representation type, then $R$ is virtually Gorenstein. 
\end{cor}

The rest of the section is devoted to the proof of Theorem \ref{Main-finite CM type}, which will be presented at the end. 
\begin{chunk}\label{finitistic}
    Let $R$ be a commutative Noetherian ring. The \emph{finitistic projective dimension} of $R$, denoted by $\FPD(R)$, is the number ${\rm sup}\{\pd_R(M)\mid M\in \Mod(R) \text{ with }\pd_R(M)<\infty\}$. By \cite[Corollary 5.5]{Bass} and \cite[Theorem 3.2.6]{RG}, 
    $$
    \FPD(R)=\dim(R).
    $$
    It is known that $R$ has finite Krull dimension if, in addition, $R$ is local or $R$ has a dualizing complex. 

Assume that $\FPD(R)$ is finite, Jensen \cite[Proposition 6]{Jensen} observed that any flat $R$-module has finite projective dimension, and hence a module with finite flat dimension has finite projective dimension.
\end{chunk}
The following result characterizes the property $\GProj(R)=\Proj(R)$. 
\begin{lem}\label{compre big and small}
    Let $R$ be a commutative Noetherian ring with finite Krull dimension. Then the following are equivalent.
\begin{enumerate}
    \item $\GProj(R)=\Proj(R)$.

   % \item Every Gorenstein projective $R$-module is a direct sum of finitely generated projective $R$-modules.

    %  \item Every Gorenstein projective $R$-module is a filtered colimit of finitely generated projective $R$-modules.

    \item $\Gproj(R)=\proj(R)$ and every Gorenstein projective $R$-module is a direct sum of finitely generated Gorenstein projective $R$-modules.

    \item $\Gproj(R)=\proj(R)$ and every Gorenstein projective $R$-module is a filtered colimit of finitely generated Gorenstein projective $R$-modules.
\end{enumerate}
\end{lem}
\begin{proof}
    $(1)\Longrightarrow (2)$. The assumption of (1) yields that $\Gproj(R)=\proj(R)$. 
Let $M$ be a Gorenstein projective $R$-module. Since $\GProj(R)=\Proj(R)$, $M$ is projective. Since $R$ is commutative Noetherian, the projective $R$-module $M$ is a direct sum of finitely generated projective $R$-modules; see \cite[Fact 3.1]{MPR} for example.

%$(3)\Longrightarrow (5)$. We only need to prove $\Gproj(R)=\proj(R)$. For each $M\in \Gproj(R)$, the assumption of $(3)$ yields that $M$ is a filtered colimit of finitely generated projective $R$-modules. This yields that $M$ is flat as the filtered colimit of flat modules is flat. Since finitely generated flat module is projective, $M\in \proj(R)$. Thus, $\Gproj(R)=\proj(R)$. 

$(2)\Longrightarrow (3)$ is trivial.

$(3)\Longrightarrow (1)$. Let $M$ be a Gorenstein projective $R$-module. The assumption yields that $M$ is a filtered colimit of finitely generated Gorenstein projective $R$-modules. Since $\Gproj(R)=\proj(R)$, each finitely generated projective is projective, and hence flat. We conclude that $M$ is a filtered colimit of finitely generated flat $R$-modules. It follows that $M$ is a flat $R$-module. By $\FPD(R)<\infty$, $M$ has finite projective dimension; see \ref{finitistic}. Combining with that $M$ is Gorenstein projective, we get that $M$ is projective. It follows that $\GProj(R)=\Proj(R)$.
\end{proof}

\begin{rem}
    Let $R$ be a commutative Artinian ring.  
    
(1) By \cite[Corollary 4.11]{Beligiannis:adv}, $\GProj(R)=\Proj(R)$ is equivalent that $\Gproj(R)=\proj(R)$ and $R$ is virtually Gorenstein. This can be also proved by combining Lemma \ref{compre big and small} with \cite[Theorem 5]{Beligiannis-Krause}

(2)  Assume, in addition, $R$ is virtually Gorenstein, it follows from \cite[Theorem 5]{Beligiannis-Krause} that $\Gproj(R)$ is contravariantly finite in $\mod(R)$. By \cite[Theorem C]{CPST}, $R$ is either Gorenstein or $\Gproj(R)=\proj(R)$. 
\end{rem}
By the above remark, we get:
\begin{cor}
    Let $R$ be a commutative Artinian ring which is not self-injective. Then $R$ is virtually Gorenstein if and only if $\GProj(R)=\Proj(R)$.
\end{cor}
% \com{Is the converse of \Cref{descent} true? Let $R$ be local CM with dualziing module $\omega_R$. If $D_{\sg}(R)=\thick(\omega_{R})$, then is $D_{\sg}(\widehat R)=\thick(\omega_{\widehat R})$ (this is perhaps true if $\widehat R$ is isolated singularity)?  }

\begin{lem}\label{descent}
    Let $R$ be a commutative Noetherian local ring. If $\GProj(\widehat R)=\Proj(\widehat R)$, then $\GProj(R)=\Proj(R)$. 
\end{lem}
\begin{proof}
    Let $M$ be a Gorenstein projective $R$-module. By \ref{finitistic} and \cite[Proposition 3.4]{Holm}, $M$ is Gorenstein flat over $R$. Hence, $\widehat{R}\otimes_R M$ is Gorenstein flat over $\widehat R$; see \cite[Proposition 3.10]{Holm}. It follows from \cite[Theorem I]{CFH} that $\widehat{R}\otimes_R M$ has finite Gorenstein projective dimension over $\widehat R$. By assumption, each Gorenstein projective $\widehat R$-module is projective over $\widehat R$, and hence $\widehat R\otimes_R M$ has finite projective dimension over $\widehat R$. In particular, $\widehat R\otimes_R M$ has finite flat dimension over $\widehat R$. This yields that $M$ has finite flat dimension over $R$ because $\widehat R$ is faithfully flat as an $R$-module. By \ref{finitistic}, $M$ has finite projective dimension over $R$. Combining with that $M$ is Gorenstein projective over $R$, we get that $M$ is projective over $R$. This finishes the proof. 
\end{proof}

The following result is a syzygy version of a result of Herzog \cite{Herzog}.
\begin{lem}\label{herzog} 
Let $R$ be a Gorenstein local ring. If $\add\Omega^n\CM_0(R)$ has finite representation type for some $n\geq 0$, then $R$ is a hypersurface. 
\end{lem}

\begin{proof} 
Under the assumption that $R$ is a Gorenstein local ring, we can verify directly that $\add\Omega^n\CM_0(R)=\CM_0(R)$. Then the assumption will yield that $\CM_0(R)$ has finite representation type. By \cite[Corollary 1.2]{DT2015}, $R$ has an isolated singularity (see also \Cref{application}), and hence $R$ has finite Cohen--Macaulay representation type. Combining this with
\cite[Theorem 9.15]{LW}, we get that $R$ is a hypersurface.
\end{proof}

For each $M\in\mod(R)$, let $\res(M)$ denote the smallest resolving subcategory containing $M$.

\begin{prop}\label{CM big G-regular}
    Let $R$ be a Cohen--Macaulay local ring. If $\CM_0(R)$ has finite representation type, then either $R$ is a hypersurface or $\GProj(R)=\Proj(R)$. 
\end{prop}
\begin{proof}
Since $\CM_0(R)$ has finite representation type. Theorem \ref{cm category version} implies that $\CM_0(\widehat{R})$ has finite representation type. On the other hand, if $\widehat{R}$ is a hypersurface, then $R$ is a hypersurface. Moreover, by Lemma \ref{descent}, if $\GProj(\widehat R)=\Proj(\widehat R)$, then $\GProj(R)=\Proj(R)$. Thus, we may assume $R$ is complete, namely $R=\widehat{R}$. Under the complete assumption, $R$ has a canonical module, denoted by $\omega$; see \cite[Corollary 3.3.8]{BH}. 

Assume $R$ is not a hypersurface. Next, we prove $\GProj(R)=\Proj(R)$.  By Corollary \ref{application}, $R$ has at isolated singularity, and hence $\CM(R)=\CM_0(R)$.  Combining with the assumption that $\CM_0(R)$ has finite representation type,  we get that $R$ is not a Gorenstein ring as $R$ is not a hypersurface; see \Cref{herzog}.  It follows that $\res_{\mod(R)}(\omega)\neq \add R$. Hence, \cite[Corollary 6.9]{class} implies that $\CM(R)=\res(\omega)$. Combining this with the fact that sufficiently high syzygies of any module belong to $\CM(R)$, 
we conclude that $\D^f(R)=\thick_{\D^f(R)}(R, \omega)$.  Combining with \ref{IK-equivalence}, 
we get that $\underline{\GProj}(R)^c=0$. This implies that $\underline{\GProj}(R)=0$ as $\underline{\GProj}(R)$ is compactly generated; see \ref{IK-equivalence}. This completes the proof.
\end{proof}

\begin{prop}\label{CM finite equivalent}
    Let $R$ be a complete Cohen--Macaulay local ring. If $\GProj(R)\neq \Proj(R)$, then the following are equivalent.

\begin{enumerate}
    \item $\CM(R)$ has finite representation type.
     \item $\CM_0(R)$ has finite representation type.

     \item $R$ is a hypersurface and every Gorenstein projective $R$-module is a direct sum of finitely generated Gorenstein projective $R$-modules.

    \item $R$ is Gorenstein and every Gorenstein projective $R$-module is a filtered colimit of finitely generated projective $R$-modules.
\end{enumerate}
\end{prop}
\begin{proof}
    $(1)\Longrightarrow (2)$ and $(3)\Longrightarrow (4)$ are trivial. 

    $(2)\Longrightarrow (3)$. Assume $\CM_0(R)$ has finite representation type. Proposition \ref{CM big G-regular}  yields that $R$ is either a hypersurface or $\GProj(R)=\Proj(R)$. By the assumption that $\GProj(R)\neq \Proj(R)$, we conclude that $R$ is a hypersurface, and hence $R$ is Gorenstein. We claim that $\Gproj(R)\neq \proj(R)$. If not, assume that $\Gproj(R)=\proj(R)$. Since $R$ is Gorenstein, for each $M\in\mod(R)$, every high enough syzygy of $M$ is Gorenstein projective, and hence every high enough syzygy of $M$ is projective. Thus, $R$ is regular. This yields that $\GProj(R)=\Proj(R)$. This contradicts with the assumption that $\GProj(R)\neq \Proj(R)$. Thus, the claim $\Gproj(R)\neq \proj(R)$ follows. Combining with that $R$ is Gorenstein, $\CM(R)=\Gproj(R)$. By \cite[Theorem 4.20]{Beligiannis:adv}, we get that every Gorenstein projective $R$-module is a direct sum of finitely generated Gorenstein projective $R$-modules. 

    $(4)\Longrightarrow (1)$. Since $R$ is Gorenstein, we conlcude as $(2)\Longrightarrow (3)$ that $\Gproj(R)\neq\proj(R)$. By \cite[Theorem 4.20]{Beligiannis:adv}, $\Gproj(R)$ has finite representation type. Note that $\Gproj(R)=\CM(R)$. We conclude that $\CM(R)$ has finite representation type. 
\end{proof}

%\begin{rem}
%    Let $R$ be a Cohen--Macaulay local ring with minimal multiplicity and canonical module $\omega$, where $k$ is infinite. If $R$ is not Gorenstein, then it follows from \cite[Lemma 5.1]{Takahashi:2008} that $\gproj(R)=\proj(R)$.
%\end{rem}

\begin{proof}[Proof of Theorem \ref{Main-finite CM type}]

%For the second statement, assume any of the equivalent conditions hold. By Theorem \ref{cm category version}, $\CM_0(\widehat{R})$ has finite representation type.  It follows from Proposition \ref{CM big G-regular} that either $\widehat R$ is a hypersurface or $\GProj(\widehat R)=\Proj(\widehat R)$. Combining with Lemma \ref{descent}, this yields that either $R$ is a hypersurface or $\GProj(R)=\Proj(R)$. 
(1) This follows from Proposition \ref{CM big G-regular}.

(2) Assume $\CM_0(R)$ has finite representation type. If $\GProj(\widehat R)=\Proj(\widehat R)$, then it follows from Lemma \ref{compre big and small} that every Gorenstein projective $\widehat R$-module is a direct sum of finitely generated Gorenstein projective $\widehat R$-modules. Assume now $\GProj(\widehat R)\neq \Proj(\widehat R)$. Since $\CM_0(R)$ has finite representation type, it follows from Theorem \ref{cm category version} that $\CM_0(\widehat R)$ has finite representation type. 
Hence,
Proposition \ref{CM finite equivalent} yields that every Gorenstein projective $\widehat R$-module is a direct sum of finitely generated Gorenstein projective $\widehat R$-modules.

(3) Assume $\GProj(R)\neq \Proj(R)$. Combining this with Lemma \ref{descent}, we have $\GProj(\widehat R)\neq \Proj(\widehat R)$. By Theorem \ref{cm category version}, $\CM(\widehat{R})$ has finite representation type if and only if so does $\CM(R)$, which is further equivalent to $\CM_0(R)$ having finite representation type by Remark \ref{all are equivalent}. The desired equivalences in Theorem \ref{Main-finite CM type} (3) now follow from Proposition \ref{CM finite equivalent}.
\end{proof}

\section{Applications on dominant local rings}\label{Section-application}
In this section, we investigate dominant local rings introduced by Takahashi \cite{Takahashi:2023}. The main result in this section is Theorem \ref{application-dominant}. It provides a class of virtually Gorenstein rings.
\begin{chunk}
  For a commutative Noetherian ring $R$,  the \emph{singularity category} of $R$ is defined as the Verdier quotient
$$\D_{\sg}(R) \colonequals \D^f (R)/\thick_{\D^f(R)}(R).$$
This category was introduced by Buchweitz \cite[5, Def. 1.2.2]{Buch} under the name ``stable derived category"; see also \cite{Orlov}.  The name ``singularity category" is justified by the fact that $R$ is regular if and only if $\D_{\sg}(R)=0$.
\end{chunk}
\begin{chunk}
A commutative Noetherian local ring $(R, \m, k)$ is said to be \emph{dominant}, as introduced by Takahashi \cite[Corollary 10.8]{Takahashi:2023}, if for each nonzero complex $M\in \D_{\sg}(R)$, $k\in \thick_{\D_{\sg}(R)}(M)$. This definition is equivalent to that, for each $M\in \D^f(R)$ with $\pd_R(M)=\infty$, $k\in \thick_{\D^f(R)}(R,M)$.
\end{chunk}

% \textbf{Question:} Should we generalize the definition of dominant local rings to non-local ring? 

\begin{ex}\label{minimal multiplicity}
(1) Let $(R,\m)$ be a commutative Artinian local ring with $\m^2=0$. Then  $R$ is dominant.

(2) By \cite[Proposition 5.10]{Takahashi:2023}, a hypersurface is a dominant local ring. This can be proved by the classification of thick subcategories over hypersurfaces given in \cite[Theorem 6.6]{Takahashi:2010}.

(3) Let $R$ be a Cohen--Macaulay local ring with minimal multiplicity and an infinite residue field. It follows from \cite[Proposition 5.10]{Takahashi:2023} that $R$ is dominant.
\end{ex}

\begin{lem}\label{equal}
    Let $R$ be a dominant local ring with an isolated singularity. For each nonzero complex $M$ in $\D_{\sg}(R)$, then
    $$
\D^f(R)=\thick_{\D^f(R)}(R,M)\text{ and } ~\D_{\sg}(R)=\thick_{\D_{\sg}(R)}(M).
$$
\end{lem}
\begin{proof}
    Two results in the statement are equivalent. We prove the latter one. Let $k$ denote the residue field of $R$.  Since $R$ has an isolated singularity, $\D_{\sg}(R)=\thick_{\D_{\sg}(R)}(k)$; see \cite[Proposition A.2]{KMVdB}. By definition of the dominant local ring, $k\in\thick_{\D_{\sg}(R)}(M)$.
    Thus, $\D_{\sg}(R)=\thick_{\D_{\sg}(R)}(M)$. 
\end{proof}
\begin{lem}\label{criteria}
    Let $R$ be a commutative Noetherian local ring with a dualizing complex. If $R$ is dominant with an isolated singularity, then either $R$ is Gorenstein or $\GProj(R)=\Proj(R)$.
\end{lem}

\begin{proof}
 Let $D$ be a dualizing complex of $R$. Assume $R$ is not Gorenstein. Then $D$ is not zero in $\D_{\sg}(R)$; see \cite[Theorem 3.3.4]{Christensen:2000}. It follows from 
  Lemma \ref{equal} that
  $$
  \D^f(R)=\thick_{\D^f(R)}(R,D).
  $$
  Combining with \ref{IK-equivalence},
we conclude that $\underline{\GProj}(R)^c=0$. This implies that $\underline{\GProj}(R)=0$ as $\underline{\GProj}(R)$ is compactly generated; see \ref{IK-equivalence}. Hence, $\GProj(R)=\Proj(R)$.
\end{proof}

\begin{rem}
    Let $R$ be a Cohen--Macaulay local ring with canonical module $\omega$.  The conclusion of Lemma \ref{criteria} remains valid if the assumption that $R$ is dominant is weakened to the condition that $R$ is quasi-dominant in the sense of \cite[Definition 10.14]{Takahashi:2023}.  Note that if, in addition, $R$ is almost Gorenstein but not Gorenstein and the residue field of $R$ is infinite, then $R$ is quasi-dominant; see \cite[Proposition 10.15]{Takahashi:2023}.
\end{rem}
\begin{cor}\label{conditionforvirtuallyG}
    Let $R$ be a commutative Noetherian local ring with a dualizing complex. If $R$ is dominant with an isolated singularity, then $R$ is virtually Gorenstein.
\end{cor}

\begin{proof}
    By Lemma \ref{criteria},  either $R$ is Gorenstein or $\GProj(R)=\Proj(R)$. If $R$ is Gorenstein, then \ref{defofvirtuallyG} yields that $R$ is virtually Gorenstein. Assume $\GProj(R)=\Proj(R)$, then it follows immediately from Lemma \ref{freeimpliesvirtuallyG} that $R$ is virtually Gorenstein. 
\end{proof}
\begin{rem}
(1) The dominance assumption on $R$ in Proposition \ref{conditionforvirtuallyG} cannot be dropped. Indeed, there exists an Artinian local ring $(R,\m)$ with $\m^3=0$ which is not virtually Gorenstein; see \cite[Example 4.3]{Beligiannis-Krause}. 

(2) By Corollary \ref{conditionforvirtuallyG} and (1), there exists an Artinian local ring $(R,\m)$ with $\m^3=0$ which is not dominant. For an Artinian local ring with $\m^3=0$, if $\embdim(R)=2$ and $R$ is not a complete intersection, then $R$ is dominant; see \cite[Proposition 8.1]{Takahashi:2023}. Here, $\embdim(R)\colonequals \dim_k(\m/\m^2)$ represents the embedding dimension of $R$.
\end{rem}

Combining Corollary \ref{application} with Example \ref{minimal multiplicity}, Corollary \ref{conditionforvirtuallyG} yields the following result.
\begin{cor}
     Let $(R,\m,k)$ be a Cohen--Macaulay local ring with minimal multiplicity and canonical module $\omega$, where $k$ is infinite. If $\OCM_0(R)$ has finite representation type, then 
     $R$ is virtually Gorenstein.
\end{cor}
\begin{prop}\label{Gorenstein-Gorenstein free}
Let $R$ be a commutative Noetherian local ring. If $R$ is dominant and $\widehat R$ has an isolated singularity, then either $R$ is Gorenstein or $\GProj(R)=\Proj(R)$.
\end{prop}

\begin{proof}
     Since $R$ is dominant, it follows from \cite[Corollary 5.8]{Takahashi:2023} that $\widehat{R}$ is dominant. Note that $\widehat R$ admits a dualizing complex. Combining with Lemma \ref{criteria}, $\widehat{R}$ is either a Gorenstein or $\GProj(\widehat R)=\Proj(\widehat R)$. If $\widehat R$ is Gorenstein, then so is $R$; see \cite[Proposition 3.1.19]{BH}. If $\GProj(\widehat R)=\Proj(\widehat R)$, then Lemma \ref{descent} yields that $\GProj(R)=\Proj(R)$. This completes the proof.
\end{proof}
\begin{rem}\label{rem-dom}
    For a local ring $R$, it is proved in \cite[Corollary 5.8]{Takahashi:2023} that $R$ is dominant if and only if $\widehat R$ is dominant. Thus, the assumption that $R$ is dominant can be replaced by that $\widehat R$ is dominant.
\end{rem}

\begin{thm}\label{application-dominant}
    Let $R$ be a dominant local ring. Assume $\TF_n^0(R)$ has finite representation type for some $n\leq \depth(R)+1$. Then:
    \begin{enumerate}
        \item Either $R$ is a  hypersurface or $\GProj(R)=\Proj(R)$.

        \item  $\widehat R$ is virtually Grorenstein.
    \end{enumerate}
\end{thm}  
\begin{proof}
  By Theorem \ref{main}, $\widehat R$ has an isolated singularity. In view of Remark \ref{rem-dom}, $\widehat R$ is dominant. It follows from Proposition \ref{Gorenstein-Gorenstein free} that either $\widehat R$ is Gorenstein or $\GProj(\widehat R)=\Proj(\widehat R)$. Combining this with Lemma \ref{freeimpliesvirtuallyG}, we conclude that $\widehat R$ is virtually Gorenstein. Hence, (2) holds.
  
As shown above, either $\widehat R$ is Gorenstein or $\GProj(\widehat R)=\Proj(\widehat R)$, this yields that either $R$ is Gorenstein or $\GProj(R)=\Proj(R)$; see \ref{descent}. It remains to show that $R$ is a hypersurface when it is Gorenstein.

Assume $R$ is Gorenstein. By 
    Proposition \ref{ttrf}, we have $\TF_d^0(R)=\CM_0(R)$, where $d=\dim(R)$. This yields the first equality below:
    $$
    \OCM_0(R)= \Omega\TF_{d}^0(R)=\TF_{d+1}^0(R)\subseteq \TF_n^0(R),
    $$
   where the second equality is also by Proposition \ref{ttrf}, and the inclusion is because $n\leq \depth(R)+1=d+1$. 
    Since $\OCM_0(R) =\add\Omega\CM_0(R) $ (see \ref{syzygy-res-closed-summands}),  \Cref{herzog} implies that $R$ is a hypersurface.  This completes the proof.  
    \end{proof}

\begin{prop}\label{several applications}
    Let $(R,\m,k)$ be a Cohen--Macaulay local ring. Assume $R$ satisfies one of the following three conditions:
    \begin{enumerate}
        \item $R$ is dominant.

        \item $R$ has minimal multiplicity, and  $k$ is an infinite field.

        \item $\dim(R)=2$, and $k$ is an algebraically closed field. \end{enumerate}
 If, in addition, $\OCM_0(R)$ has finite representation type, then 
     $R$ is either a hypersurface or $\GProj(R)=\Proj(R)$.
\end{prop}
\begin{proof}
Assume (3) holds. By Proposition \ref{dim2}, $R$ has minimal mulitiplicity.  Thus, the condition (3) implies the condition (2). Assume (2) holds.  By Example \ref{minimal multiplicity}, $R$ is dominant, and hence the condition (1) holds. It remains to show that the desired result holds under the assumption of (1).   

Assume (1) holds. By Proposition \ref{ttrf}, $\TF_{d+1}^0(R)=\OCM_0(R)$, where $d=\dim(R)$. Since $\OCM_0(R)$ has finite representation type, it follows from Theorem \ref{application-dominant} that either $R$ is a hypersurface or $\GProj(R)=\Proj(R)$.
\end{proof}

In general, it is not easy to determine whether a non-Gorenstein ring is virtually Gorenstein. The following result identifies several classes of rings that are virtually Gorenstein. It builds on earlier results as well as the work of Takahashi \cite{Takahashi:2023}. 

\begin{cor}\label{classofrings-VG}
Let $(R,\m,k)$ be a commutative Noetherian local ring. Then $R$ is virtually Gorenstein if $R$ satisfies one of the following conditions:

(1) $R$ is Artinian and is dominant. 

(2) $R$ has an isolated singularity and $R$ is of the form $Q/I$ with $\mu(I)\leq 2$, where $Q$ is a regular local ring and $I$ is a proper ideal of $Q$.

(3) $\m^3=0$ and $\embdim(R)=2$. 

(4) $R=k\llbracket x,y\rrbracket /(x^{a_1},x^{a_2}y^{b_2},\ldots,x^{a_{n-1}}y^{b_{n-1}},y^{b_n})$, where $a_1>a_2>\cdots>a_{n-1}>a_n=0$ and $0=b_1<b_2<\cdots<b_{n-1}<b_n$ are integers with $n\geq 3$.
 
\end{cor}
\begin{proof}
All these conditions imply that $R$ has a dualizing complex. 

    (1) This follows from Corollary 
 \ref{conditionforvirtuallyG}.

    (2) By \cite[Corollary 8.9]{Takahashi:2023}, either $R$ is  a complete intersection or a dominant local ring. If $R$ is a complete intersection, then \ref{defofvirtuallyG} yields that $R$ is virtually Gorenstein. Assume that $R$ is a dominant local ring. Combining with the assumption that $R$ has an isolated singularity, Corollary \ref{conditionforvirtuallyG} yields that $R$ is virtually Gorenstein. 

    (3) By \cite[Proposition 8.1]{Takahashi:2023}, either $R$ is a complete intersection or a dominant local ring. Then the same argument as (2) yields that $R$ is virtually Gorenstein. 

    (4) By \cite[Corollary 8.4]{Takahashi:2023}, $R$ is dominant. It is clear that $R$ has an isolated singularity. The desired result now follows from Corollary \ref{conditionforvirtuallyG}.
\end{proof}

\begin{ex}
(1) Let $R$ be an Artinian local ring with $\m^2=0$, then \cite[Example 3.13]{ZAD} yields that $R$ is virtually Gorenstein; this can be also proved by combining Example \ref{minimal multiplicity} with Corollary \ref{classofrings-VG} (1).
\end{ex}

(2) Let $k$ be a field and $R=k\llbracket x,y\rrbracket/(x^n, xy^m)$ for some $n,m>0$. Note that $\Spec(R)=\{(x),(x,y)\}$ and $R_{(x)}$ is field. Thus, $R$ has an isolated singularity. By Corollary \ref{classofrings-VG} (2), $R$ is virtually Gorenstein. 

(3) Let $k$ be a field. By Corollary \ref{classofrings-VG} (3) or (4), $k\llbracket x,y\rrbracket/(x^3,xy,y^3)$ is virtually Gorenstein.

\bibliographystyle{amsplain}
\bibliography{ref}

@Book{BH,
 Author = {Bruns, Winfried and Herzog, J{\"u}rgen},
 Title = {Cohen-{Macaulay} rings},
 Edition = {Rev. ed.},
 FSeries = {Cambridge Studies in Advanced Mathematics},
 Series = {Camb. Stud. Adv. Math.},
 Volume = {39},
 ISBN = {0-521-56674-6},
 Year = {1998},
 Publisher = {Cambridge: Cambridge University Press},
 zbMATH = {1194481},
 Zbl = {0909.13005}
}

@article{Takahashi-uni-dom,
 Author = {Takahashi, Ryo},
 Title = {Uniformly dominant local rings and {Orlov} spectra of singularity categories},
 Year = {2024},
 HowPublished = {Preprint, {arXiv}:2412.00669 [math.{AC}] (2024)},
 pages = {\url{https://arxiv.org/abs/2412.00669}},
 journal = {arXiv:2412.00669
}}

@Book{Christensen:2000,
 Author = {Christensen, Lars Winther},
 Title = {Gorenstein dimensions},
 FSeries = {Lecture Notes in Mathematics},
 Series = {Lect. Notes Math.},
 ISSN = {0075-8434},
 Volume = {1747},
 ISBN = {3-540-41132-1},
 Year = {2000},
 Publisher = {Berlin: Springer},
 DOI = {10.1007/BFb0103980},
 zbMATH = {1544059},
 Zbl = {0965.13010}
}

@article{Dey-Takahashi2022,
 author = {Dey, Souvik and Takahashi, Ryo},
 title = {Comparisons between annihilators of {Tor} and {Ext}},
 fjournal = {Acta Mathematica Vietnamica},
 journal = {Acta Math. Vietnam.},
 issn = {0251-4184},
 volume = {47},
 number = {1},
 pages = {123--139},
 year = {2022},
 doi = {10.1007/s40306-021-00443-0},
 zbMATH = {7507853},
 Zbl = {1485.13032}
}

@Article{Dey-Takahashi,
 Author = {Dey, Souvik and Takahashi, Ryo},
 Title = {On the subcategories of {{\(n\)}}-torsionfree modules and related modules},
 FJournal = {Collectanea Mathematica},
 Journal = {Collect. Math.},
 ISSN = {0010-0757},
 Volume = {74},
 Number = {1},
 Pages = {113--132},
 Year = {2023},
 DOI = {10.1007/s13348-021-00338-1},
 zbMATH = {7640070},
 Zbl = {1524.13050}
}

@Article{DT2014,
 Author = {Dao, Hailong and Takahashi, Ryo},
 Title = {The radius of a subcategory of modules},
 FJournal = {Algebra \& Number Theory},
 Journal = {Algebra Number Theory},
 ISSN = {1937-0652},
 Volume = {8},
 Number = {1},
 Pages = {141--172},
 Year = {2014},
 DOI = {10.2140/ant.2014.8.141},
 zbMATH = {6322074},
 Zbl = {1308.13015}
}

@Article{DT2015,
 Author = {Dao, Hailong and Takahashi, Ryo},
 Title = {The dimension of a subcategory of modules},
 FJournal = {Forum of Mathematics, Sigma},
 Journal = {Forum Math. Sigma},
 ISSN = {2050-5094},
 Volume = {3},
 Pages = {31},
 Note = {Id/No e19},
 Year = {2015},
 DOI = {10.1017/fms.2015.19},
 zbMATH = {6502816},
 Zbl = {1353.13011}
}

@Article{Iyengar-krause,
 Author = {Iyengar, Srikanth B. and Krause, Henning},
 Title = {Acyclicity versus total acyclicity for complexes over noetherian rings},
 FJournal = {Documenta Mathematica},
 Journal = {Doc. Math.},
 ISSN = {1431-0635},
 Volume = {11},
 Pages = {207--240},
 Year = {2006},
 zbMATH = {5037004},
 Zbl = {1119.13014}
}

@Article{IT2016,
 Author = {Iyengar, Srikanth B. and Takahashi, Ryo},
 Title = {Annihilation of cohomology and strong generation of module categories},
 FJournal = {IMRN. International Mathematics Research Notices},
 Journal = {Int. Math. Res. Not.},
 ISSN = {1073-7928},
 Volume = {2016},
 Number = {2},
 Pages = {499--535},
 Year = {2016},
 DOI = {10.1093/imrn/rnv136},
 zbMATH = {6561542},
 Zbl = {1355.13015}
}

@Article{KMVdB,
 Author = {Keller, Bernhard and Murfet, Daniel and Van den Bergh, Michel},
 Title = {On two examples by {Iyama} and {Yoshino}},
 FJournal = {Compositio Mathematica},
 Journal = {Compos. Math.},
 ISSN = {0010-437X},
 Volume = {147},
 Pages = {591--612},
 Year = {2011},
 DOI = {10.1112/S0010437X10004902},
 zbMATH = {5882544},
 Zbl = {1264.13016}
}

@article {groth,
    AUTHOR = {Dao, Hailong and Iyama, Osamu and Takahashi, Ryo and Vial,
              Charles},
     TITLE = {Non-commutative resolutions and {G}rothendieck groups},
   JOURNAL = {J. Noncommut. Geom.},
  FJOURNAL = {Journal of Noncommutative Geometry},
    VOLUME = {9},
      YEAR = {2015},
    NUMBER = {1},
     PAGES = {21--34},
      ISSN = {1661-6952},
   MRCLASS = {14E15 (13D15)},
  MRNUMBER = {3337953},
MRREVIEWER = {David A. Jorgensen},
       DOI = {10.4171/JNCG/186},
       URL = {https://doi.org/10.4171/JNCG/186},
}

@article {class,
    AUTHOR = {Takahashi, Ryo},
     TITLE = {Classifying resolving subcategories over a {C}ohen-{M}acaulay
              local ring},
   JOURNAL = {Math. Z.},
  FJOURNAL = {Mathematische Zeitschrift},
    VOLUME = {273},
      YEAR = {2013},
    NUMBER = {1-2},
     PAGES = {569--587},
      ISSN = {0025-5874,1432-1823},
   MRCLASS = {13C60 (13C14 16G60)},
MRREVIEWER = {Branden\ R.\ Stone},
       DOI = {10.1007/s00209-012-1020-1},
       URL = {https://doi.org/10.1007/s00209-012-1020-1},
}

@Article{Takahashi:2014,
 Author = {Takahashi, Ryo},
 Title = {Reconstruction from {Koszul} homology and applications to module and derived categories},
 FJournal = {Pacific Journal of Mathematics},
 Journal = {Pac. J. Math.},
 ISSN = {1945-5844},
 Volume = {268},
 Number = {1},
 Pages = {231--248},
 Year = {2014},
 DOI = {10.2140/pjm.2014.268.231},
 zbMATH = {6323670},
 Zbl = {1309.18013}
}

@Book{LW,
 Author = {Leuschke, Graham and Wiegand, Roger},
 Title = {Cohen-{Macaulay} representations},
 FSeries = {Mathematical Surveys and Monographs},
 Series = {Math. Surv. Monogr.},
 ISSN = {0076-5376},
 Volume = {181},
 ISBN = {978-0-8218-7581-0},
 Year = {2012},
 Publisher = {Providence, RI: American Mathematical Society (AMS)},
 zbMATH = {6039453},
 Zbl = {1252.13001}
}

@article{LW2000,
 author = {Leuschke, Graham and Wiegand, Roger},
 title = {Ascent of finite {Cohen}-{Macaulay} type},
 fjournal = {Journal of Algebra},
 journal = {J. Algebra},
 issn = {0021-8693},
 volume = {228},
 number = {2},
 pages = {674--681},
 year = {2000},
 doi = {10.1006/jabr.2000.8294},
 zbMATH = {1471836},
 Zbl = {0963.13020}
}

@Article{Takahashi:2010,
 Author = {Takahashi, Ryo},
 Title = {Classifying thick subcategories of the stable category of {Cohen}-{Macaulay} modules},
 FJournal = {Advances in Mathematics},
 Journal = {Adv. Math.},
 ISSN = {0001-8708},
 Volume = {225},
 Number = {4},
 Pages = {2076--2116},
 Year = {2010},
 DOI = {10.1016/j.aim.2010.04.009},
 zbMATH = {5796757},
 Zbl = {1202.13009}
}

@Article{Takahashi:2023,
 Author = {Takahashi, Ryo},
 Title = {Dominant local rings and subcategory classification},
 FJournal = {IMRN. International Mathematics Research Notices},
 Journal = {Int. Math. Res. Not.},
 ISSN = {1073-7928},
 Volume = {2023},
 Number = {9},
 Pages = {7259--7318},
 Year = {2023},
 DOI = {10.1093/imrn/rnac053},
 zbMATH = {7711406},
 Zbl = {1528.13014}
}

@Article{Iyengar-Krause2022,
 Author = {Iyengar, Srikanth B. and Krause, Henning},
 Title = {The {Nakayama} functor and its completion for {Gorenstein} algebras},
 FJournal = {Bulletin de la Soci{\'e}t{\'e} Math{\'e}matique de France},
 Journal = {Bull. Soc. Math. Fr.},
 ISSN = {0037-9484},
 Volume = {150},
 Number = {2},
 Pages = {347--391},
 Year = {2022},
 DOI = {10.24033/bsmf.2849},
 zbMATH = {7594380},
 Zbl = {1505.16020}
}

@Book{Enochs-Jenda,
 Author = {Enochs, Edgar E. and Jenda, Overtoun M. G.},
 Title = {Relative homological algebra. {Vol}. 2},
 Edition = {2nd revised ed.},
 FSeries = {De Gruyter Expositions in Mathematics},
 Series = {De Gruyter Expo. Math.},
 ISSN = {0938-6572},
 Volume = {54},
 ISBN = {978-3-11-021522-9; 978-3-11-021523-6},
 Year = {2011},
 Publisher = {Berlin: Walter de Gruyter},
 zbMATH = {5818169},
 Zbl = {1238.13002}
}

@Article{Beligiannis,
 Author = {Beligiannis, Apostolos},
 Title = {Cohen-{Macaulay} modules, (co)torsion pairs and virtually {Gorenstein} algebras},
 FJournal = {Journal of Algebra},
 Journal = {J. Algebra},
 ISSN = {0021-8693},
 Volume = {288},
 Number = {1},
 Pages = {137--211},
 Year = {2005},
 DOI = {10.1016/j.jalgebra.2005.02.022},
 zbMATH = {2174886},
 Zbl = {1119.16007}
}

@Article{Beligiannis:adv,
 Author = {Beligiannis, Apostolos},
 Title = {On algebras of finite {Cohen}-{Macaulay} type},
 FJournal = {Advances in Mathematics},
 Journal = {Adv. Math.},
 ISSN = {0001-8708},
 Volume = {226},
 Number = {2},
 Pages = {1973--2019},
 Year = {2011},
 DOI = {10.1016/j.aim.2010.09.006},
 zbMATH = {5835553},
 Zbl = {1239.16016}
}

@article{Kimura-compact,
 author = {Kimura, Kaito},
 title = {Compactness of the {Alexandrov} topology of maximal {Cohen}-{Macaulay} modules},
 fjournal = {Communications in Algebra},
 journal = {Commun. Algebra},
 issn = {0092-7872},
 volume = {53},
 number = {8},
 pages = {3536--3549},
 year = {2025},
 doi = {10.1080/00927872.2025.2462278},
 zbMATH = {8055024}
}

@article{Kimura,
 author = {Kimura, Kaito},
 title = {Stability of annihilators of cohomology and closed subsets defined by Jacobian ideals},
 journal = {arXiv:2409.17934},
 issn = {0092-7872},
 pages = {\url{https://arxiv.org/abs/2409.17934}},
 year = {2024}
}

@Article{Holm,
 Author = {Holm, Henrik},
 Title = {Gorenstein homological dimensions.},
 FJournal = {Journal of Pure and Applied Algebra},
 Journal = {J. Pure Appl. Algebra},
 ISSN = {0022-4049},
 Volume = {189},
 Number = {1-3},
 Pages = {167--193},
 Year = {2004},
 DOI = {10.1016/j.jpaa.2003.11.007},
 URL = {curis.ku.dk/ws/files/41927598/GorensteinHomologicalDimensions.pdf},
 zbMATH = {2078857},
 Zbl = {1050.16003}
}

@Article{CFH,
 Author = {Christensen, Lars Winther and Frankild, Anders and Holm, Henrik},
 Title = {On {Gorenstein} projective, injective and flat dimensions -- a functorial description with applications},
 FJournal = {Journal of Algebra},
 Journal = {J. Algebra},
 ISSN = {0021-8693},
 Volume = {302},
 Number = {1},
 Pages = {231--279},
 Year = {2006},
 DOI = {10.1016/j.jalgebra.2005.12.007},
 zbMATH = {5072885},
 Zbl = {1104.13008}
}

@Article{Bass,
 Author = {Bass, Hyman},
 Title = {Injective dimension in {Noetherian} rings},
 FJournal = {Transactions of the American Mathematical Society},
 Journal = {Trans. Am. Math. Soc.},
 ISSN = {0002-9947},
 Volume = {102},
 Pages = {18--29},
 Year = {1962},
 DOI = {10.2307/1993878},
 zbMATH = {3204765},
 Zbl = {0126.06503}
}

@Article{RG,
 Author = {Raynaud, Michel and Gruson, Laurent},
 Title = {Crit{\`e}res de platitude et de projectivit{\'e}. {Techniques} de ''platification'' d'un module. ({Criterial} of flatness and projectivity. {Technics} of ''flatification of a module.)},
 FJournal = {Inventiones Mathematicae},
 Journal = {Invent. Math.},
 ISSN = {0020-9910},
 Volume = {13},
 Pages = {1--89},
 Year = {1971},
 DOI = {10.1007/BF01390094},
 URL = {https://eudml.org/doc/142084},
 zbMATH = {3360302},
 Zbl = {0227.14010}
}

@Article{Jensen,
 Author = {Jensen, Christian U.},
 Title = {On the vanishing of {{\(\lim_{(i)}\)}}},
 FJournal = {Journal of Algebra},
 Journal = {J. Algebra},
 ISSN = {0021-8693},
 Volume = {15},
 Pages = {151--166},
 Year = {1970},
 Language = {English},
 DOI = {10.1016/0021-8693(70)90071-2},
 zbMATH = {3317887},
 Zbl = {0199.36202}
}

@Article{BHST,
 Author = {Bahlekeh, Abdolnaser and Hakimian, Ehsan and Salarian, Shokrollah and Takahashi, Ryo},
 Title = {Annihilation of cohomology, generation of modules and finiteness of derived dimension},
 FJournal = {The Quarterly Journal of Mathematics},
 Journal = {Q. J. Math.},
 ISSN = {0033-5606},
 Volume = {67},
 Number = {3},
 Pages = {387--404},
 Year = {2016},
 DOI = {10.1093/qmath/haw015},
 zbMATH = {6654443},
 Zbl = {1368.13014}
}

@Article{Beligiannis-Krause,
 Author = {Beligiannis, Apostolos and Krause, Henning},
 Title = {Thick subcategories and virtually {Gorenstein} algebras.},
 FJournal = {Illinois Journal of Mathematics},
 Journal = {Ill. J. Math.},
 ISSN = {0019-2082},
 Volume = {52},
 Number = {2},
 Pages = {551--562},
 Year = {2008},
 zbMATH = {5615653},
 Zbl = {1200.16022}
}

@Article{CPST,
 Author = {Christensen, Lars Winther and Piepmeyer, Greg and Striuli, Janet and Takahashi, Ryo},
 Title = {Finite {Gorenstein} representation type implies simple singularity},
 FJournal = {Advances in Mathematics},
 Journal = {Adv. Math.},
 ISSN = {0001-8708},
 Volume = {218},
 Number = {4},
 Pages = {1012--1026},
 Year = {2008},
 DOI = {10.1016/j.aim.2008.03.005},
 zbMATH = {5285706},
 Zbl = {1148.14004}
}

@Article{MPR,
 Author = {McGovern, Warren Wm. and Puninski, Gena and Rothmaler, Philipp},
 Title = {When every projective module is a direct sum of finitely generated modules},
 FJournal = {Journal of Algebra},
 Journal = {J. Algebra},
 ISSN = {0021-8693},
 Volume = {315},
 Number = {1},
 Pages = {454--481},
 Year = {2007},
 DOI = {10.1016/j.jalgebra.2007.01.043},
 zbMATH = {5201693},
 Zbl = {1126.13008}
}

@Article{BFS,
 Author = {Bahlekeh, Abdolnaser and Fotouhi, Fahimeh Sadat and Salarian, Shokrollah},
 Title = {Representation-theoretic properties of balanced big {Cohen}-{Macaulay} modules},
 FJournal = {Mathematische Zeitschrift},
 Journal = {Math. Z.},
 ISSN = {0025-5874},
 Volume = {293},
 Number = {3-4},
 Pages = {1673--1709},
 Year = {2019},
 DOI = {10.1007/s00209-019-02257-1},
 zbMATH = {7126341},
 Zbl = {1430.13015}
}

@article{Mifune2024,
 Author = {Mifune, Yuki},
 Title = {A generalization of the dimension and radius of a subcategory of modules and its applications},
 Year = {2025},
journal={Rend. Sem. Mat. Univ. Padova},
pages={published online first},
 HowPublished = {Preprint, {arXiv}:2401.11153 [math.{AC}] (2024)},
pages = {\url{https://arxiv.org/abs/2401.11153}},
 Journal = {arXiv:2401.11153}
}

@Article{Wiegand,
 Author = {Wiegand, Roger},
 Title = {Local rings of finite {Cohen}-{Macaulay} type},
 FJournal = {Journal of Algebra},
 Journal = {J. Algebra},
 ISSN = {0021-8693},
 Volume = {203},
 Number = {1},
 Pages = {156--168},
 Year = {1998},
 DOI = {10.1006/jabr.1997.7319},
 zbMATH = {1203910},
 Zbl = {0921.13015}
}

@Article{MatT,
 Author = {Matsui, Hiroki and Takahashi, Ryo},
 Title = {Maximal {C}ohen–{M}acaulay Approximations and {S}erre’s Condition},
 FJournal = {Acta Math Vietnam},
 Journal = {Acta Math Vietnam},
 Volume = {40},
 Pages = {197--203},
 Year = {2015}
}

@book{AB,
 author = {Auslander,  Maurice and Bridger, Mark},
 title = {Stable module theory},
 fseries = {Memoirs of the American Mathematical Society},
 series = {Mem. Am. Math. Soc.},
 issn = {0065-9266},
 volume = {94},
 isbn = {978-0-8218-1294-5; 978-1-4704-0044-6},
 year = {1969},
 publisher = {Providence, RI: American Mathematical Society (AMS)},
 doi = {10.1090/memo/0094},
 zbMATH = {3324867},
 Zbl = {0204.36402}
}

@article{Masek,
 author = {Ma{\c{s}}ek, Vladimir},
 title = {Gorenstein dimension and torsion of modules over commutative {Noetherian} rings},
 fjournal = {Communications in Algebra},
 journal = {Commun. Algebra},
 issn = {0092-7872},
 volume = {28},
 number = {12},
 pages = {5783--5811},
 year = {2000},
language = {English},
 doi = {10.1080/00927870008827189},
 zbMATH = {1584360},
 Zbl = {1002.13005}
}

@article{KLT,
 author = {Kobayashi, Toshinori and Lyle, Justin and Takahashi, Ryo},
 title = {Maximal {Cohen}-{Macaulay} modules that are not locally free on the punctured spectrum},
 fjournal = {Journal of Pure and Applied Algebra},
 journal = {J. Pure Appl. Algebra},
 issn = {0022-4049},
 volume = {224},
 number = {7},
 pages = {29},
 note = {Id/No 106311},
 year = {2020},
 doi = {10.1016/j.jpaa.2020.106311},
 zbMATH = {7173212},
 Zbl = {1455.13023}
}

@article{DS,
 author = {Dibaei, Mohammad T. and Sadeghi, Arash},
 title = {Linkage of modules and the {Serre} conditions},
 fjournal = {Journal of Pure and Applied Algebra},
 journal = {J. Pure Appl. Algebra},
 issn = {0022-4049},
 volume = {219},
 number = {10},
 pages = {4458--4478},
 year = {2015},
 doi = {10.1016/j.jpaa.2015.02.027},
 zbMATH = {6444752},
 Zbl = {1317.13023}
}

@misc{Sch,
 author = {Schreyer, Frank-Olaf},
 title = {Finite and countable {CM}-representation type},
 year = {1987},
 howpublished = {Singularities, representation of algebras, and vector bundles, {Proc}. {Symp}., {Lambrecht}/{Pfalz}/{FRG} 1985, {Lect}. {Notes} {Math}. 1273, 9-34 (1987).},
 zbMATH = {4185806},
 Zbl = {0719.14024}
}

@article{Chen-arXiv,
 author = {Chen, Xiao-Wu},
 title = {Gorenstein {Homological} {Algebra} of {Artin} {Algebras}},
 year = {2017},
 howpublished = {Preprint, {arXiv}:1712.04587 [math.{RT}] (2017)},
pages= {\url{https://arxiv.org/abs/1712.04587}},
journal={arXiv:1712.04587}
}

@article{Chen:2008,
 author = {Chen, Xiao-Wu},
 title = {An {Auslander}-type result for {Gorenstein}-projective modules},
 fjournal = {Advances in Mathematics},
 journal = {Adv. Math.},
 issn = {0001-8708},
 volume = {218},
 number = {6},
 pages = {2043--2050},
 year = {2008},
 doi = {10.1016/j.aim.2008.04.004},
 zbMATH = {5312139},
 Zbl = {1147.16015}
}

@misc{Auslander:1982,
 author = {Auslander, Maurice},
 title = {A functorial approach to representation theory},
 year = {1982},
 howpublished = {Representations of algebras, 3rd int. {Conf}., {Puebla}/{Mex}. 1980, {Lect}. {Notes} {Math}. 944, 105-179},
 zbMATH = {3764123},
 Zbl = {0486.16027}
}

@article{Auslander:1974,
 author = {Auslander, Maurice},
 title = {Representation theory of {Artin} algebras. {II}},
 fjournal = {Communications in Algebra},
 journal = {Commun. Algebra},
 issn = {0092-7872},
 volume = {1},
 pages = {269--310},
 year = {1974},
 doi = {10.1080/00927877409412807},
 zbMATH = {3447096},
 Zbl = {0285.16029}
}

@article{Herzog,
 author = {Herzog, J{\"u}rgen},
 title = {Rings mit only finitely many isomorphism classes of maximal indecomposable {Cohen}-{Macaulay} modules},
 fjournal = {Mathematische Annalen},
 journal = {Math. Ann.},
 issn = {0025-5831},
 volume = {233},
 pages = {21--34},
 year = {1978},
 doi = {10.1007/BF01351494},
 url = {https://eudml.org/doc/163092},
 zbMATH = {3557914},
 Zbl = {0358.13009}
}

@article{Krause,
 author = {Krause, Henning},
 title = {The stable derived category of a noetherian scheme},
 fjournal = {Compositio Mathematica},
 journal = {Compos. Math.},
 issn = {0010-437X},
 volume = {141},
 number = {5},
 pages = {1128--1162},
 year = {2005},
 doi = {10.1112/S0010437X05001375},
 zbMATH = {2212628},
 Zbl = {1090.18006}
}

@book{Beligiannis-Reiten,
 author = {Beligiannis, Apostolos and Reiten, Idun},
 title = {Homological and homotopical aspects of torsion theories},
 fseries = {Memoirs of the American Mathematical Society},
 series = {Mem. Am. Math. Soc.},
 issn = {0065-9266},
 volume = {883},
 isbn = {978-0-8218-3996-6; 978-1-4704-0487-1},
 year = {2007},
 publisher = {Providence, RI: American Mathematical Society (AMS)},
 doi = {10.1090/memo/0883},
 url = {semanticscholar.org/paper/a8f0a5d62f17b0c03106f6fb3a978dc89dd01b7d},
 zbMATH = {5179140},
 Zbl = {1124.18005}
}

@article{ZAD,
 author = {Zareh-Khoshchehreh, Fatemeh and Asgharzadeh, Mohsen and Divaani-Aazar, Kamran},
 title = {Gorenstein homology, relative pure homology and virtually {Gorenstein} rings},
 fjournal = {Journal of Pure and Applied Algebra},
 journal = {J. Pure Appl. Algebra},
 issn = {0022-4049},
 volume = {218},
 number = {12},
 pages = {2356--2366},
 year = {2014},
 doi = {10.1016/j.jpaa.2014.04.005},
 zbMATH = {6315950},
 Zbl = {1297.13017}
}

@article{DLW,
 author = {Di, Zhenxing and Liang, Li and Wang, Junpeng},
 title = {Virtually {Gorenstein} rings and relative homology of complexes},
 fjournal = {Journal of Pure and Applied Algebra},
 journal = {J. Pure Appl. Algebra},
 issn = {0022-4049},
 volume = {227},
 number = {1},
 pages = {15},
 note = {Id/No 107127},
 year = {2023},
 doi = {10.1016/j.jpaa.2022.107127},
 zbMATH = {7567641},
 Zbl = {1497.18013}
}

@Book{Buch,
 Author = {Buchweitz, Ragnar-Olaf},
 Title = {Maximal {Cohen}-{Macaulay} modules and {Tate} cohomology. {With} appendices by {Luchezar} {L}. {Avramov}, {Benjamin} {Briggs}, {Srikanth} {B}. {Iyengar} and {Janina} {C}. {Letz}},
 FSeries = {Mathematical Surveys and Monographs},
 Series = {Math. Surv. Monogr.},
 ISSN = {0076-5376},
 Volume = {262},
 ISBN = {978-1-4704-5340-4; 978-1-4704-6792-0},
 Year = {2021},
 Publisher = {Providence, RI: American Mathematical Society (AMS)},
 DOI = {10.1090/surv/262},
 zbMATH = {7498869},
 Zbl = {1505.13002}
}

@InCollection{Orlov,
 Author = {Orlov, Dmitri},
 Title = {Triangulated categories of singularities and {D}-branes in {Landau}-{Ginzburg} models},
 BookTitle = {Algebraic geometry. Methods, relations, and applications. Collected papers. Dedicated to the memory of Andrei Nikolaevich Tyurin.},
 Pages = {227--248},
 Year = {2004},
 Publisher = {Moscow: Maik Nauka/Interperiodica},
 zbMATH = {5071248},
 Zbl = {1101.81093}
}

@article{DLMO,
 author = {Dey, Souvik and Liu, Jian and Mifune, Yuki and Otake, Yuya},
 title = {Generation of singularity categories and infinite injective dimension locus via annihilation of cohomologies},
 year = {2025},
 pages = {\url{https://arxiv.org/abs/2503.24186}},
 journal= {arXiv:2503.24186}
}

@article{DLT,
 author = {Dey, Souvik and Lank, Pat and Takahashi, Ryo},
 title = {Strong generation for module categories},
fjournal = {Journal of Pure and Applied Algebra},
journal={J. Pure Appl. Algebra},
volume = {229},
number = {10},
pages = {108070},
year = {2025},
 howpublished = {Preprint, {arXiv}:2307.13675 [math.{AC}] (2023)}
}

@misc{Auslander:1986,
 author = {Auslander, Maurice},
 title = {Isolated singularities and existence of almost split sequences. {Notes} by {Louise} {Unger}},
 year = {1986},
 howpublished = {Representation theory {II}, {Groups} and orders, {Proc}. 4th {Int}. {Conf}., {Ottawa}/{Can}. 1984, {Lect}. {Notes} {Math}. 1178, 194-242},
 zbMATH = {4029718},
 Zbl = {0633.13007}
}

@article{HL,
 author = {Huneke, Craig and Leuschke, Graham J.},
 title = {Two theorems about maximal {Cohen}-{Macaulay} modules},
 fjournal = {Mathematische Annalen},
 journal = {Math. Ann.},
 issn = {0025-5831},
 volume = {324},
 number = {2},
 pages = {391--404},
 year = {2002},
 doi = {10.1007/s00208-002-0343-3},
 url = {surface.syr.edu/cgi/viewcontent.cgi?article=1078&context=mat},
 zbMATH = {1837394},
 Zbl = {1007.13005}
}

@article{Kno,
 author = {Kn{\"o}rrer, Horst},
 title = {Cohen-{Macaulay} modules on hypersurface singularities. {I}},
 fjournal = {Inventiones Mathematicae},
 journal = {Invent. Math.},
 issn = {0020-9910},
 volume = {88},
 pages = {153--164},
 year = {1987},
 doi = {10.1007/BF01405095},
 url = {https://eudml.org/doc/143448},
 zbMATH = {4000120},
 Zbl = {0617.14033}
}

@Article{BGS,
 Author = {Buchweitz, Ragnar-Olaf and Greuel, Gert-Martin and Schreyer,Frank-Olaf},
 Title = {Cohen-{Macaulay} modules on hypersurface singularities. {II}},
 FJournal = {Inventiones Mathematicae},
 Journal = {Invent. Math.},
 ISSN = {0020-9910},
 Volume = {88},
 Pages = {165--182},
 Year = {1987},
 DOI = {10.1007/BF01405096},
 zbMATH = {4000121},
 Zbl = {0617.14034}
}

@book{Yoshino,
 author = {Yoshino, Yuji},
 title = {Cohen-{Macaulay} modules over {Cohen}-{Macaulay} rings},
 fseries = {London Mathematical Society Lecture Note Series},
 series = {Lond. Math. Soc. Lect. Note Ser.},
 issn = {0076-0552},
 volume = {146},
 isbn = {0-521-35694-6},
 year = {1990},
 publisher = {Cambridge (UK): Cambridge University Press},
 zbMATH = {44599},
 Zbl = {0745.13003}
}

@article{EH,
 author = {Eisenbud, David and Herzog, Jürgen},
 title = {The classification of homogeneous {Cohen}-{Macaulay} rings of finite representation type},
 fjournal = {Mathematische Annalen},
 journal = {Math. Ann.},
 issn = {0025-5831},
 volume = {280},
 number = {2},
 pages = {347--352},
 year = {1988},
 doi = {10.1007/BF01456058},
 url = {https://eudml.org/doc/164367},
 zbMATH = {3997953},
 Zbl = {0616.13011}
}
\end{document}